\documentclass[a4paper,10pt,fleqn]{article}
\usepackage{amsmath,amssymb,amsthm,graphicx,subfigure,float,caption,epstopdf,tabularx,color, bm,amsfonts,epic}
\usepackage[top=1.25in, bottom=1.0in, left=1.0in, right=1.0in]{geometry}
\usepackage{appendix}
\usepackage{amssymb}
\usepackage{multirow}
\usepackage{mathrsfs}
\usepackage[table]{xcolor}
\usepackage[numbers,sort&compress]{natbib}
\usepackage{comment,enumerate,multicol,xspace}
\usepackage{fancyhdr}
\usepackage{soul}
\usepackage{enumerate}
\usepackage{makecell}
\definecolor{Cobalt}{rgb}{0.25,0.41,0.88}
\usepackage[linkcolor=Cobalt,anchorcolor=Cobalt,citecolor=Cobalt,colorlinks]{hyperref}
\usepackage{bbm}
\usepackage{algorithm}
\usepackage{algorithmic}
\usepackage{booktabs}
\usepackage{lineno}
\usepackage{dsfont}

\usepackage{pifont}
\usepackage{cases}

\allowdisplaybreaks
\newtheorem{thm}{Theorem}[section]
\newtheorem{lem}{Lemma}[section]

\numberwithin{equation}{section}

\usepackage{indentfirst}
\graphicspath{{Fig}}

\begin{document}
\title{Numerical simulation of the dual-phase-lag heat conduction equation on a one-dimensional unbounded domain using artificial boundary condition}

\author{Weiping Bu$^{1,2}$\footnote{Correspondence author. Email: weipingbu@xtu.edu.cn},\quad Zhengfang Xie$^1$,\quad Yushi Wang$^1$
\\{\small $^1$ School of Mathematics and Computational Science, Xiangtan University, }
\\{\small Hunan 411105, China}
\\{\small $^2$ Hunan Key Laboratory for Computation and Simulation in Science and}
\\{\small Engineering, Hunan 411105, China}
}

\date{}
\maketitle
\begin{abstract}
This paper focuses on the numerical solution of a dual-phase-lag heat conduction equation on a space unbounded domain. First, based on the Laplace transform and the Pad\'e approximation,
a high-order local artificial boundary condition is constructed for the considered  problem, which effectively transforms the original problem into an initial-boundary value problem on a bounded computational domain. Subsequently, for the resulting reduced problem on the bounded domain equipped with high-order local artificial boundary, a stability result based on the $L^2$-norm is derived. Next, we develop finite difference method for the reduced problem by introducing auxiliary variable to reduce the order of time derivative. The numerical analysis demonstrates that the developed numerical scheme is unconditionally stable and possesses a second-order convergence rate in both space and time. Finally, numerical results are presented to validate the effectiveness of the proposed numerical method and the correctness of the theoretical analysis.
\\[2ex]
\end{abstract}


 \vskip 5mm
\section{Introduction}\label{Sec1}
In classical heat conduction theory, Fourier's Law has been successfully and widely applied to describe the relationship between heat flux and temperature gradient. However, this law will yield an infinite speed of heat propagation,
which is generally inconsistent with physical reality. To overcome this inherent drawback, various alternative models have been proposed. For instance, Cattaneo \cite{Cattaneo1958form} and Vernotte \cite{Vernotte1958Les,Vernotte1961Some} propose a modified Fourier's Law by introducing a first-order time derivative of the heat flux.
Subsequently, Tzou \cite{Tzou1989Shock,Tzou1990Thermalb} extends this result and develops the single-phase-lag (SPL) heat conduction model.

Since the single-phase-lag heat conduction model only accounts for heat flux lag, to more accurately describe micro-scale heat transfer phenomena, Tzou \cite{Tzou2014Macro-to} not only considers the lag of heat flux but also that of the temperature gradient, further extending  into the following dual-phase-lag (DPL) heat conduction model
\begin{equation}\label{eqs1_2}
q(x, t + \tau_q) = -k \nabla T(x, t + \tau_T),
\end{equation}
where $q$ is heat flux, $k>0$ is thermal condutivity, $T$ is the temperature, $\nabla$ is the gradient operator, and $\tau_q>0, \tau_T>0$ are the phase lags of
the heat flux vector and the temperature gradient, respectively.
What is more, a Taylor's expansion of (\ref{eqs1_2}) readily yields the following DPL heat conduction equation \cite{Cheng2008Single-and}
$$\frac{\partial T(x, t)}{\partial t} + \tau_q \frac{\partial^2 T(x, t)}{\partial t^2} = \alpha \Delta T(x, t) + \tau_T \alpha \frac{\partial}{\partial t} (\Delta T(x, t)).$$
To date, this model has been widely applied in various fields including bioheat transfer, ultrafast laser heating, micro/nano-scale heat conduction, and thermal damage in biological tissues \cite{Xu2015Time,Shomali2022Lagging,Ghazanfarian2015Macro-to,Ghazanfarian2009Effect,Hobiny2020Nonlinear,Qiao2021Numerical,Lee2013Analysis}.
In this paper, for convenience, we denote
$\frac{\partial T}{\partial t}, \frac{\partial^2 T}{\partial t^2}, \frac{\partial^2 T}{\partial x^2}$ and $ \frac{\partial}{\partial t}\left(\frac{\partial^2 T}{\partial x^2}\right)$ as
$T_t, T_{tt}, T_{xx}$ and $T_{txx}$, respectively. We then proceed to discuss the following dimensionless DPL heat conduction equation on a one-dimensional semi-infinite domain
\begin{equation}
\left\{ \begin{aligned}
&T_t(x,t)+aT_{tt}(x,t)=K(T_{xx}(x,t)+bT_{txx}(x,t))+f(x,t),
(x,t)\in\Omega\times(0,D],\\
&T(x,0)=\xi(x),x\in\Omega,\\
&T_t(x,0)=\eta(x),x\in\Omega,\\
&T(0,t)=\phi(t),t\in(0,D],\\
&T(x,t)\rightarrow0,x\rightarrow+\infty,t\in(0,D],
\end{aligned}\right.\label{eqs1_3}
\end{equation}
where $\Omega:=\{x|0\leq x<+\infty\},f(x,t),\xi(x),\eta(x),\phi(t)$ are some given continuous functions, $a,b,K$ are some given positive constants. Let $x_r>0$ be a constant. Here, we suppose that the support of $f(x,t)$ is $\Omega_{2}\times[0,D]$ with $\Omega_2:=[0,x_r],$
and the supports of
$\xi(x)$ and $\eta(x)$ are on the domain $\Omega_{2}$.

Recently, while analytical solutions for DPL heat conduction equations have been explored such as \cite{Wang2001Well-posedness,Noroozi2016new,Julius2018Nonhomogeneous,Youssef2020exact}, their inherent complexity, often involving sophisticated special functions, precludes their direct practical implementation. Therefore, considerable research efforts have been dedicated to developing numerical approaches for these equations.
In \cite{Cabrera2013Difference}, Cabrera et al. employed the Crank-Nicolson finite difference method to solve two types of DPL models, and analyzed the stability and convergence of the numerical scheme. By employing finite difference and finite element wavelet Galerkin methods to solve DPL bioheat transfer models, Kumar et al. \cite{Kumar2016Numerical} theoretically investigated the thermal behavior of living biological tissues during thermal therapy under various coordinate systems and non-Fourier boundary conditions.
In \cite{Majchrzak2020Numerical}, Majchrzak and Mochnacki
proposed both explicit and implicit finite difference methods for solving a class of second-order DPL heat conduction model. By using the weighted shifted Gr\"unwald difference formula in the temporal direction and the Legendre spectral method in the spatial direction, Zheng et al. \cite{Zheng2020Efficient} solved a class of time-fractional DPL heat conduction models, and discussed the stability and convergence of the numerical scheme.
In \cite{Wang2020Analytical}, Wang et al. utilized the L1/finite difference method to obtain numerical solution for a class of Caputo fractional DPL heat conduction models during short-pulse laser heating, and investigated the solvability, stability, and convergence of the numerical scheme.
In \cite{Kumar2025Computational}, Kumar and Meena solved SPL and DPL models by adopting Gaussian radial basis functions and Chebyshev polynomials. In \cite{Ansari2025Numerical}, Ansari et al. obtained numerical solution for a class of DPL nonlinear bioheat transfer models by employing finite difference method in the spatial direction and Runge-Kutta method in the temporal direction.

For the numerical solution of partial differential equations (PDEs) on unbounded domains, early research commonly employs Dirichlet or Neumann boundary conditions on a finite domain directly as substitutes for conditions at infinity. However, these traditional boundary conditions are inherently only crude approximations of the exact boundary conditions. Consequently, the construction of exact or highly accurate approximate artificial boundary conditions (ABCs) has emerged as a crucial problem for solving PDEs on unbounded domains \cite{Han2013Artificial}.
Up to now, several studies have numerically solved the heat equation and diffusion equation by using artificial boundary method (ABM).
In \cite{Han2002class}, Han and Huang obtained exact ABCs for the heat conduction equation on unbounded domains, and solved the reduced problem with finite difference and finite element methods.
In \cite{Wu2004Convergence}, Wu and Sun employed ABM and finite difference method to solve the heat conduction equation on a two-dimensional unbounded domain, and discussed the unconditional stability and convergence of the numerical scheme.
Antoine et al. \cite{Antoine2008review} provided a systematic review of ABM for solving linear and nonlinear Schr\"odinger equations on unbounded domains. For one-dimensional unbounded fractional diffusion equations, Gao et al. \cite{Gao2012finite} obtained numerical solution by using ABM combined with the L1/finite difference method, and established the stability and convergence theory for the proposed numerical scheme.
For time-fractional diffusion-wave equations on two- and three-dimensional unbounded spatial domains, Brunner et al. \cite{Brunner2014Artificial} derived ABCs of the considered problem, obtained numerical solution, and developed the associated numerical theory.
In \cite{Li2014Local}, Li et al. constructed high-order local ABCs for the Schr\"odinger and heat equations. Subsequently, they solved these reduced problems by using finite difference method.
For nonlinear time-fractional parabolic equations, Li and Zhang \cite{Li2016Efficient} derived nonlinear absorbing boundary conditions, and solved the reduced nonlinear problem by using a linearized finite difference method.
By utilizing the Pad\'e approximation, Zhang et al. \cite{Zhang2017Artificial} constructed high-order local ABCs for a class of non-local heat equations and performed numerical solution.
Based on the Pad\'e approximation, Dong et al. \cite{Dong2020High-order} constructed high-order local ABCs for a class of fractional diffusion-wave equations, and used the L1/finite difference method to solve the reduced problem. By employing the Z-transform, Sun et al. \cite{Sun2020Fast} constructed ABCs for one-dimensional convection-diffusion equations on unbounded domains and performed numerical solution.
By introducing fractional derivative, Li et al. \cite{Li2022Fast} derived exact ABCs for one-dimensional reaction-convection-diffusion equations on unbounded domain, and developed an unconditionally stable fast finite difference scheme.
In \cite{Liu2023Analysis}, Liu et al. constructed exact absorbing boundary conditions for a fractional anomalous diffusion model in a comb structure on unbounded domains, and proposed corresponding finite difference method.
In \cite{Zhu2024highly}, Zhu and Xu developed highly efficient numerical method for the time
fractional diffusion equation on unbounded domains by using ABCs, L2 formula and sum-of-exponentials technique.
For a time fractional anomalous diffusion equations on unbounded domains, Su et al. \cite{Su2025Numerical} devised absorbing boundary conditions, and established stable and convergent numerical scheme.

It can be seen from the above discussion that current numerical studies on the DPL heat conduction equations primarily focus on bounded domains, while research on unbounded domains is almost nonexistent. In view of that ABM is an efficient and well-established approach for solving PDEs on unbounded domains, this motivates us to employ ABM to solve the DPL heat conduction equation (\ref{eqs1_3}) on unbounded domain. The main contribution of this paper includes:

$\bullet$ Based on the Laplace transform and the Pad\'e approximation, a high-order local artificial
boundary for problem (\ref{eqs1_3}) is derived.

$\bullet$ For the reduced problem of (\ref{eqs1_3}) with the high-order local artificial boundary, its stability is discussed and a priori estimate under the $L^2$-norm is obtained.

$\bullet$ By introducing auxiliary variable to reduce the order of the time derivative of the reduced problem, we construct finite difference scheme for the reduced problem, and discuss the stability and convergence of the developed numerical scheme.

$\bullet$ Some numerical tests are conducted to validate the correctness of the obtained numerical theory.

The structure of this paper is as follows. In Section 2, we derive ABCs for problem (\ref{eqs1_3}). In Section 3, we investigate the stability of the associated reduced problem. In Section 4, finite difference scheme of the reduced problem is developed, and the corresponding stability and convergence theory is established. In Section 5, we present several numerical tests to validate the theoretical results. Finally, a summary is provided in Section 6.

\section{ABCs for the problem (\ref{eqs1_3})}\label{Sec2}

In this section, we devise the ABCs for the considered problem (\ref{eqs1_3}) and investigate the stability of the developed reduced problem.
In order to transform the initial-boundary value problem (\ref{eqs1_3}) on a semi-infinite domain into a
DPL heat equation on a finite computational domain, we
employ ABM.
The artificial boundary can be taken as $\Gamma=\{(x,t)|x=x_r,0\leq t\leq D\},$ which divides the semi-infinite domain $\Omega$ into a bounded domain $\Omega_2:=[0,x_r]$ and an unbounded domain $\Omega_1:=(x_r,+\infty)$.

\subsection{Construction of ABCs}
Now we consider the problem on the unbounded domain $\Omega_1\times(0,D]$ as follows
\begin{subequations}
\begin{numcases}{}
T_t(x,t)+aT_{tt}(x,t)=K(T_{xx}(x,t)+bT_{txx}(x,t)),(x,t)\in\Omega_1\times(0,D],\label{eqs2_2.1a}\\
T(x,0)=0,x\in\Omega_1,\label{eqs2_2.1b}\\
T_t(x,0)=0,x\in\Omega_1,\label{eqs2_2.1c}\\
T(x,t)\rightarrow0,x\rightarrow+\infty,t\in(0,D].\label{eqs2_2.1d}
\end{numcases}
\end{subequations}
Applying the Laplace transform to (\ref{eqs2_2.1a}) yields
$$s\hat{T}(x,s)+as^2\hat{T}(x,s)=K(\hat{T}_{xx}(x,s)+bs\hat{T}_{xx}(x,s)),$$
i.e.,
\begin{equation}
(s+as^2)\hat{T}(x,s)=K(1+bs)\hat{T}_{xx}(x,s),\label{eqs2_2}
\end{equation}
where $\hat{T}(x,s)$ denotes the Laplace transform of $T(x,t)$.
By solving the second-order ordinary differential equation (\ref{eqs2_2}) with respect to
$x$, and noting that  (\ref{eqs2_2.1d}) implies $|\hat{T}(x,s)|<+\infty, x\rightarrow+\infty$, we have
\begin{equation}
\hat{T}(x,s)=c_1(s)e^{-\sqrt{\frac{s+as^2}{K(1+bs)}}x},x\in(x_r,+\infty).\label{eqs2_3}
\end{equation}
Differentiating both sides of (\ref{eqs2_3}) with respect to $x$ gives
\begin{equation}
(1+bs)\hat{T}_x(x,s)+\sqrt{\frac{(1+bs)(1+as)}{s}}\frac{s}{\sqrt{K}}\hat{T}(x,s)=0.\label{eqs2_4}
\end{equation}
Let $z=\frac{(1+bs)(1+as)}{s}$. By using the Pad\'e expansion
$$
\sqrt{z} \approx \sqrt{z_0}\left(1-\sum_{n=1}^N \frac{a_n\left(1-z / z_0\right)}{1-b_n\left(1-z / z_0\right)}\right)
$$
to replace $\sqrt{z}$ in (\ref{eqs2_4}), we obtain
\begin{equation}
(1+bs)\hat{T}_x(x,s)+\sqrt{z_0}\left(1-\sum_{n=1}^N \frac{a_n\left(z_0-z\right)}{z_0-b_n\left(z_0-z\right)}\right)\frac{s}{\sqrt{K}}\hat{T}(x,s)=0,\label{eqs2_5}
\end{equation}
where the parameter $z_0>0$ is the Pad\'e expansion point, and
$$
b_n=\cos ^2\left(\frac{n \pi}{2 N+1}\right), a_n=\frac{2}{2 N+1} \sin ^2\left(\frac{n \pi}{2 N+1}\right).
$$
By introducing the auxiliary function
\begin{equation}\label{eqs_250924A}
\hat{\zeta}_n( s)=\frac{1}{z_0-b_n\left(z_0-z\right)}  \hat{T}\left(x_r, s\right),
\end{equation}
then, when $x \rightarrow x_r$, (\ref{eqs2_5}) can be equivalently expressed as
\begin{equation}
\left\{ \begin{aligned}
&(1+bs)\hat{T}_x(x_r,s)+\sqrt{\frac{z_0}{K}}\left[\left(1+\sum_{n=1}^N \frac{a_n}{b_n}\right) s\hat{T}(x_r,s)-z_0 \sum_{n=1}^N \frac{a_n}{b_n} s\hat{\zeta}_n( s)\right]=0, \\
&\left(z_0-b_n z_0+b_n z\right) \hat{\zeta}_n( s)=\hat{T}(x_r,s).
\end{aligned}\right.\label{eqs2_6}
\end{equation}
Applying the inverse Laplace transform to (\ref{eqs2_6}) with $z=\frac{(1+bs)(1+as)}{s}$, we obtain
\begin{equation}
\left\{ \begin{aligned}
&T_x(x_r,t)+bT_{tx}(x_r,t)+\sqrt{\frac{z_0}{K}}\left[\left(1+\sum_{n=1}^N \frac{a_n}{b_n}\right) T_t(x_r,t)-z_0 \sum_{n=1}^N \frac{a_n}{b_n} \zeta^{'}_n( t)\right]=0, \\
&(z_0-b_n z_0)\zeta^{'}_n(t)+b_n[\zeta_n(t)+(a+b)\zeta_n^{'}(t)+ab\zeta^{''}_n(t)]=T_t(x_r,t),t\in(0,D].
\end{aligned}\right.\label{eqs2_7}
\end{equation}
Since
\begin{equation}\label{eqs_250924B}
s\hat{\zeta}_n( s)=\frac{1}{z_0-b_n\left(z_0-z\right)}  s\hat{T}\left(x_r, s\right),
\end{equation}
and  the initial values  $\mathrm{supp}\{\xi(x)\}\subset\Omega_{2},\mathrm{supp}\{\eta(x)\}\subset\Omega_{2}$ implies $T(x_r,0)=T_t(x_r,0)=0,$
it follows from the inverse Laplace transform, (\ref{eqs_250924A}) and (\ref{eqs_250924B}) that
\begin{equation}
\zeta_n(0)=\zeta^{'}_n(0)=0,(n=1,\dots,N).\label{eqs2_8}
\end{equation}
By combining (\ref{eqs2_7}) with (\ref{eqs2_8}), the problem (\ref{eqs1_3}) can be transformed into the following reduced problem
\begin{equation}
\left\{ \begin{aligned}
&T_t(x,t)+aT_{tt}(x,t)=K(T_{xx}(x,t)+bT_{txx}(x,t))+f(x,t),(x,t)\in\Omega_2\times(0,D],\\
&T(x,0)=\xi(x),x\in\Omega_2,\\
&T_t(x,0)=\eta(x),x\in\Omega_2,\\
&T(0,t)=\phi(t),t\in(0,D],\\
&T_x(x_r,t)+bT_{tx}(x_r,t)+\sqrt{\frac{z_0}{K}}\left[\left(1+\sum_{n=1}^N \frac{a_n}{b_n}\right) T_t(x_r,t)-z_0 \sum_{n=1}^N \frac{a_n}{b_n} \zeta^{'}_n( t)\right]=0, \\
&(z_0-b_n z_0)\zeta^{'}_n(t)+b_n[\zeta_n(t)+(a+b)\zeta_n^{'}(t)+ab\zeta^{''}_n(t)]=T_t(x_r,t),t\in(0,D],\\
&\zeta_n(0)=\zeta^{'}_n(0)=0,(n=1,\dots,N).
\end{aligned}\right.\label{eqs2_9}
\end{equation}
In order to facilitate the subsequent discussion, we let $\tilde{T}(x,t)=T(x,t)-\phi(t)$ to make sure $\tilde{T}(0,t)=0$. Then (\ref{eqs2_9}) can be  equivalently rewritten into (for simplicity, we omit  '$\tilde{\ }$', and still denote $\tilde{T}$ by $T$)
\begin{subequations}\label{eqn-250925A}
\begin{numcases}{}
T_t(x,t)+aT_{tt}(x,t)=K(T_{xx}(x,t)+bT_{txx}(x,t))+F(x,t),(x,t)\in\Omega_2\times(0,D],\label{eqs2_2.10a}\\
T(x,0)=\xi_1(x),x\in\Omega_2,\label{eqs2_2.10b}\\
T_t(x,0)=\eta_1(x),x\in\Omega_2,\label{eqs2_2.10c}\\
T(0,t)=0,t\in(0,D],\label{eqs2_2.10d}\\
T_x(x_r,t)+bT_{tx}(x_r,t)+\sqrt{\frac{z_0}{K}}\left[\left(1+\sum_{n=1}^N \frac{a_n}{b_n}\right) T_t(x_r,t)-z_0 \sum_{n=1}^N \frac{a_n}{b_n} \zeta^{'}_n( t)\right]=Q(t),\label{eqs2_2.10e}\\
(z_0-b_n z_0)\zeta^{'}_n(t)+b_n[\zeta_n(t)+(a+b)\zeta_n^{'}(t)+ab\zeta^{''}_n(t)]=T_t(x_r,t)+R(t),t\in(0,D],\label{eqs2_2.10f}\\
\zeta_n(0)=\zeta^{'}_n(0)=0,(n=1,\dots,N),\label{eqs2_2.10g}
\end{numcases}
\end{subequations}
where $F(x,t)=f(x,t)-\phi_t(t)-a\phi_{tt}(t),\xi_1(x)=\xi(x)-\phi(0)$$,\eta_1(x)=\eta(x)-\phi_t(0),Q(t)=-\sqrt{\frac{z_0}{K}}(1+\sum_{n=1}^N \frac{a_n}{b_n})\phi_t(t),$ $R(t)=\phi_t(t).$
Due to this equivalence, in the following sections, we will discuss the stability of the reduced problem (\ref{eqn-250925A}), propose corresponding numerical scheme, and establish the associated numerical theory.

\subsection{Stability of the reduced problem}
In this subsection, we investigate the stability of the reduced DPL problem (\ref{eqn-250925A}). Define the inner product and $L^2$-norm by
$$ (u,v)=\int_{0}^{x_r}u(x)v(x)dx,\quad\|u\|_{L^2(\Omega_2)}=(u,u)^{1/2}.$$
Now we state the stability of the solution to (\ref{eqn-250925A}).

\begin{thm}
Let $T(x,t)$ be the solution to (\ref{eqn-250925A}). Then
\begin{equation}
\begin{aligned}
\|T(x,t)\|^2_{L^2(\Omega_2)}
\leq&\frac{x_r^2K^{\frac{1}{2}}(2N+1)}{4\sqrt{z_0}}\int^t_0Q^2(s)ds+\frac{x_r^2z_0^{\frac{3}{2}}}{4\sqrt{K}(a+b)}\int_0^t\sum_{n=1}^N \frac{a_n}{b^2_n} R^2(s)ds+\frac{x_r^2}{2}\left\|\frac{d\xi_1(x)}{dx}\right\|^2_{L^2(\Omega_2)}
\nonumber\\
&+\frac{x_r^2}{2K}\int_0^t\|F(x,s)\|^2_{L^2(\Omega_2)}ds+\frac{ax_r^2}{2K}\|\eta_1(x)\|^2_{L^2(\Omega_2)}.
\end{aligned}
\end{equation}
\end{thm}
\begin{proof}
Multiplying both sides of (\ref{eqs2_2.10a}) by $T_t(x,t)$ and integrating it with respect to  $\Omega_2\times(0,t]$ yields
\begin{equation}\label{eqs3_1}
\begin{aligned}
&\int_0^t\|T_s(x,s)\|^2_{L^2(\Omega_2)}ds+a\int_0^t(T_{ss}(x,s),T_s(x,s))ds\\
=&K\int_0^t(T_{xx}(x,s),T_s(x,s))ds+Kb\int_0^t(T_{sxx}(x,s),T_s(x,s))ds+\int_0^t(F(x,s),T_s(x,s))ds.
\end{aligned}
\end{equation}
Note that
\begin{equation}\label{eqs3_2}
\begin{aligned}
\int_0^t(T_{ss}(x,s),T_s(x,s))ds=\frac{1}{2}\|T_t(x,t)\|^2_{L^2(\Omega_2)}-\frac{1}{2}\|\eta_1(x)\|^2_{L^2(\Omega_2)},
\end{aligned}
\end{equation}
\begin{equation}\label{eqs3_3}
\begin{aligned}
\int_0^t(T_{xx}(x,s),T_s(x,s))ds
=&\int_0^tT_{x}(x_r,s)T_s(x_r,s)ds-\frac{1}{2}\|T_x(x,t)\|^2_{L^2(\Omega_2)}+\frac{1}{2}\|(\xi_1)_x(x)\|^2_{L^2(\Omega_2)},
\end{aligned}
\end{equation}
and
\begin{equation}\label{eqs3_4}
\begin{aligned}
\int_0^t(T_{sxx}(x,s),T_s(x,s))ds
=&\int_0^tT_{sx}(x_r,s)T_s(x_r,s)ds-\int^t_0\|T_{sx}(x,s)\|^2_{L^2(\Omega_2)}ds.
\end{aligned}
\end{equation}
The substitution of (\ref{eqs3_2})--(\ref{eqs3_4}) into (\ref{eqs3_1}) results in
\begin{equation}\label{eqs3_5}
\begin{aligned}
&\int_0^t\|T_s(x,s)\|^2_{L^2(\Omega_2)}ds+\frac{a}{2}\|T_t(x,t)\|^2_{L^2(\Omega_2)}-\frac{a}{2}\|\eta_1(x)\|^2_{L^2(\Omega_2)}\\
=&K\int_0^t(T_{x}(x_r,s)+bT_{sx}(x_r,s))T_s(x_r,s)ds
-\frac{K}{2}\|T_x(x,t)\|^2_{L^2(\Omega_2)}+\frac{K}{2}\|(\xi_1)_x(x)\|^2_{L^2(\Omega_2)}\\
&-Kb\int^t_0\|T_{sx}(x,s)\|^2_{L^2(\Omega_2)}ds
+\int_0^t(F(x,s),T_s(x,s))ds.
\end{aligned}
\end{equation}
According to (\ref{eqs2_2.10e}), we replace $T_{x}(x_r,s)+bT_{sx}(x_r,s)$  which is in the integrand of the first term on the right hand side of (\ref{eqs3_5}) leading to
\begin{equation}
\begin{aligned}
&\int_0^t\|T_s(x,s)\|^2_{L^2(\Omega_2)}ds+\frac{a}{2}\|T_t(x,t)\|^2_{L^2(\Omega_2)}-\frac{a}{2}\|\eta_1(x)\|^2_{L^2(\Omega_2)}\\
=&-\sqrt{Kz_0}\int_0^t\left[\left(1+\sum_{n=1}^N \frac{a_n}{b_n}\right) T_s(x_r,s)-z_0 \sum_{n=1}^N \frac{a_n}{b_n} \zeta^{'}_n( s)\right]T_s(x_r,s)ds+K\int^t_0Q(s)T_s(x_r,s)ds\\
&-\frac{K}{2}\|T_x(x,t)\|^2_{L^2(\Omega_2)}+\frac{K}{2}\|(\xi_1)_x(x)\|^2_{L^2(\Omega_2)}-Kb\int^t_0\|T_{sx}(x,s)\|^2_{L^2(\Omega_2)}ds
+\int_0^t(F(x,s),T_s(x,s))ds.\label{eqs3_6}
\end{aligned}
\end{equation}
After performing some trivial derivations on (\ref{eqs3_6}), it is easy to obtain
\begin{equation}
\begin{aligned}
&\int_0^t\|T_s(x,s)\|^2_{L^2(\Omega_2)}ds+\frac{a}{2}\|T_t(x,t)\|^2_{L^2(\Omega_2)}-\frac{a}{2}\|\eta_1(x)\|^2_{L^2(\Omega_2)}+\sqrt{K}z_0^{\frac{3}{2}}\int_0^t\sum_{n=1}^N \frac{a_n}{b_n} \zeta^{'}_n( s)T_s(x_r,s)ds\\
=&-\sqrt{Kz_0}\int_0^t\left(1+\sum_{n=1}^N \frac{a_n}{b_n}\right) (T_s(x_r,s))^2ds+2\sqrt{K}z_0^{\frac{3}{2}}\int_0^t\sum_{n=1}^N \frac{a_n}{b_n} \zeta^{'}_n( s)T_s(x_r,s)ds\\
&-\frac{K}{2}\|T_x(x,t)\|^2_{L^2(\Omega_2)}+\frac{K}{2}\|(\xi_1)_x(x)\|^2_{L^2(\Omega_2)}-Kb\int^t_0\|T_{sx}(x,s)\|^2_{L^2(\Omega_2)}ds
+\int_0^t(F(x,s),T_s(x,s))ds\\
&+K\int^t_0Q(s)T_s(x_r,s)ds.\label{eqs3_7}
\end{aligned}
\end{equation}
Analogous to the derivation from (\ref{eqs3_5}) to (\ref{eqs3_6}), by employing (\ref{eqs2_2.10f}) to reformulate the last integral on the left hand side of (\ref{eqs3_7}), it arrives at
\begin{equation}\label{eqs3_8}
\begin{aligned}
&\int_0^t\|T_s(x,s)\|^2_{L^2(\Omega_2)}ds+\frac{a}{2}\|T_t(x,t)\|^2_{L^2(\Omega_2)}
+\sqrt{K}z_0^{\frac{3}{2}}\int_0^t\sum_{n=1}^N a_n \zeta^{'}_n( s)[\zeta_n(s)+(a+b)\zeta_n^{'}(s)+ab\zeta^{''}_n(s)]ds\\
=&\frac{a}{2}\|\eta_1(x)\|^2_{L^2(\Omega_2)}+\sqrt{K}\int_0^t\Bigg[-\sqrt{z_0}\left(1+\sum_{n=1}^N \frac{a_n}{b_n}\right) (T_s(x_r,s))^2+2z^{\frac{3}{2}}_0 \sum_{n=1}^N \frac{a_n}{b_n} \zeta^{'}_n( s)T_s(x_r,s)\\
&-z^{\frac{5}{2}}_0 \sum_{n=1}^N \frac{a_n(1-b_n)}{b_n} (\zeta^{'}_n( s))^2\bigg]ds+K\int^t_0Q(s)T_s(x_r,s)ds+\sqrt{K}z_0^{\frac{3}{2}}\int_0^t\sum_{n=1}^N \frac{a_n}{b_n} \zeta^{'}_n( s)R(s)ds\\
&-\frac{K}{2}\|T_x(x,t)\|^2_{L^2(\Omega_2)}+\frac{K}{2}\|(\xi_1)_x(x)\|^2_{L^2(\Omega_2)}-Kb\int^t_0\|T_{sx}(x,s)\|^2_{L^2(\Omega_2)}ds
+\int_0^t(F(x,s),T_s(x,s))ds.
\end{aligned}
\end{equation}
As the facts $a_n=\frac{2(1-b_n)}{2N+1}$ and
$\sum_{n=1}^N \frac{2(1-b_n)}{(2N+1)b_n}=\sum_{n=1}^N \frac{2}{(2N+1)b_n}-\sum_{n=1}^N \frac{2}{(2N+1)}$
imply
\begin{equation}\label{eqs3_9}
\begin{aligned}
&-\sqrt{z_0}\left(1+\sum_{n=1}^N \frac{a_n}{b_n}\right) (T_s(x_r,s))^2+2z^{\frac{3}{2}}_0 \sum_{n=1}^N \frac{a_n}{b_n} \zeta^{'}_n( s)T_s(x_r,s)-z^{\frac{5}{2}}_0 \sum_{n=1}^N \frac{a_n(1-b_n)}{b_n} (\zeta^{'}_n( s))^2\\
=&2z_0^{3/2} \sum_{n=1}^N \frac{2(1-b_n)}{(2N+1)b_n}T_s(x_r,s)\zeta^{'}_n( s)
-\sqrt{z_0}\left(1+\sum_{n=1}^N \frac{2(1-b_n)}{(2N+1)b_n}\right){T}_s^2(x_r,s)-z_0^{5/2}\sum _{n=1}^N\frac{2(1-b_n)^2}{(2N+1)b_n}(\zeta^{'}_n( s))^2\\
=&2z_0^{3/2} \sum_{n=1}^N \frac{2(1-b_n)}{(2N+1)b_n}T_s(x_r,s)\zeta^{'}_n( s)
-\sqrt{z_0}\sum_{n=1}^N \frac{2}{(2N+1)b_n}{T}_s^2(x_r,s)
-z_0^{5/2}\sum _{n=1}^N\frac{2(1-b_n)^2}{(2N+1)b_n}(\zeta^{'}_n( s))^2\\
&+\sqrt{z_0}\sum_{n=1}^N \frac{2}{(2N+1)}{T}_s^2(x_r,s)-\sqrt{z_0}{T}_s^2(x_r,s)\\
=&-\sum_{n=1}^N \frac{2\sqrt{z_0}}{(2N+1)b_n}\left[(1-b_n)z_0\zeta^{'}_n(s)-{T}_s(x_r,s) \right]^2- \frac{\sqrt{z_0}}{2N+1}{T}_s^2(x_r,s),
\end{aligned}
\end{equation}
by reformulating the second term on the right side of (\ref{eqs3_8}) with (\ref{eqs3_9}), we can therefore obtain
\begin{equation}\label{eqs3_10}
\begin{aligned}
&\int_0^t\|T_s(x,s)\|^2_{L^2(\Omega_2)}ds+\frac{a}{2}\|T_t(x,t)\|^2_{L^2(\Omega_2)}-\frac{a}{2}\|\eta_1(x)\|^2_{L^2(\Omega_2)}\\
&+\sqrt{K}z_0^{\frac{3}{2}}\int_0^t\sum_{n=1}^N a_n \zeta^{'}_n( s)[\zeta_n(s)+(a+b)\zeta_n^{'}(s)+ab\zeta^{''}_n(s)]ds\\
\leq&- \sqrt{K}\int_0^t\frac{\sqrt{z_0}}{2N+1}{T}_s^2(x_r,s)ds+K\int^t_0Q(s)T_s(x_r,s)ds+\sqrt{K}z_0^{\frac{3}{2}}\int_0^t\sum_{n=1}^N \frac{a_n}{b_n} \zeta^{'}_n( s)R(s)ds\\
&-\frac{K}{2}\|T_x(x,t)\|^2_{L^2(\Omega_2)}+\frac{K}{2}\|(\xi_1)_x(x)\|^2_{L^2(\Omega_2)}-Kb\int^t_0\|T_{sx}(x,s)\|^2_{L^2(\Omega_2)}ds
+\int_0^t(F(x,s),T_s(x,s))ds.
\end{aligned}
\end{equation}
Since the Cauchy-Schwarz inequality implies
\begin{equation}\label{eqs3_11}
\begin{aligned}
\int_0^t(F(x,s),T_s(x,s))ds\leq\frac{1}{2}\int_0^t\|F(x,s)\|^2_{L^2(\Omega_2)}ds+\frac{1}{2}\int_0^t\|T_s(x,s)\|^2_{L^2(\Omega_2)}ds,
\end{aligned}
\end{equation}
\begin{equation}\label{eqs3_12}
\begin{aligned}
K\int^t_0Q(s)T_s(x_r,s)ds\leq\sqrt{K}\int_0^t\frac{\sqrt{z_0}}{2N+1}{T}_s^2(x_r,s)ds+\frac{K^{\frac{3}{2}}(2N+1)}{4\sqrt{z_0}}\int^t_0Q^2(s)ds,
\end{aligned}
\end{equation}
\begin{equation}\label{eqs3_13}
\begin{aligned}
\sqrt{K}z_0^{\frac{3}{2}}\int_0^t\sum_{n=1}^N \frac{a_n}{b_n} \zeta^{'}_n( s)R(s)ds\leq\sqrt{K}z_0^{\frac{3}{2}}(a+b)\int_0^t\sum_{n=1}^N a_n (\zeta^{'}_n( s))^2ds+\frac{\sqrt{K}z_0^{\frac{3}{2}}}{4(a+b)}\int_0^t\sum_{n=1}^N \frac{a_n}{b^2_n} R^2(s)ds,
\end{aligned}
\end{equation}
it follows from (\ref{eqs3_10})--(\ref{eqs3_13}) that
\begin{equation}\label{eqs3_14}
\begin{aligned}
&\frac{1}{2}\int_0^t\|T_s(x,s)\|^2_{L^2(\Omega_2)}ds+\frac{a}{2}\|T_t(x,t)\|_{L^2(\Omega_2)}^2-\frac{a}{2}\|\eta_1(x)\|_{L^2(\Omega_2)}^2+\frac{K}{2}\|T_x(x,t)\|_{L^2(\Omega_2)}^2\\
&+\sqrt{K}z_0^{\frac{3}{2}}\int_0^t\sum_{n=1}^N a_n \zeta^{'}_n( s)[\zeta_n(s)+(a+b)\zeta_n^{'}(s)+ab\zeta^{''}_n(s)]ds\\
\leq&\frac{K^{\frac{3}{2}}(2N+1)}{4\sqrt{z_0}}\int^t_0Q^2(s)ds+\frac{\sqrt{K}z_0^{\frac{3}{2}}}{4(a+b)}\int_0^t\sum_{n=1}^N \frac{a_n}{b^2_n} R^2(s)ds+\sqrt{K}z_0^{\frac{3}{2}}(a+b)\int_0^t\sum_{n=1}^N a_n (\zeta^{'}_n( s))^2ds\\
&+\frac{K}{2}\|(\xi_1)_x(x)\|_{L^2(\Omega_2)}^2-Kb\int^t_0\|T_{sx}(x,s)\|^2_{L^2(\Omega_2)}ds
+\frac{1}{2}\int_0^t\|F(x,s)\|^2_{L^2(\Omega_2)}ds.
\end{aligned}
\end{equation}
When (\ref{eqs2_2.10g}) is employed, considering that
\begin{equation}\label{eqs3_15}
\begin{aligned}
&\int_0^t\zeta^{'}_n( s)[\zeta_n(s)+(a+b)\zeta_n^{'}(s)+ab\zeta^{''}_n(s)]ds\\
=&\int_0^t\zeta^{'}_n( s)\zeta_n(s)ds+(a+b)\int_0^t(\zeta_n^{'}(s))^2ds+ab\int_0^t\zeta_n^{'}(s)\zeta^{''}_n(s)ds\\
=&\frac{1}{2}(\zeta_n(t))^2-\frac{1}{2}(\zeta_n(0))^2+(a+b)\int_0^t(\zeta_n^{'}(s))^2ds+\frac{ab}{2}(\zeta^{'}_n(t))^2-\frac{ab}{2}(\zeta^{'}_n(0))^2\\
=&\frac{1}{2}(\zeta_n(t))^2+(a+b)\int_0^t(\zeta_n^{'}(s))^2ds+\frac{ab}{2}(\zeta^{'}_n(t))^2,
\end{aligned}
\end{equation}
(\ref{eqs3_14}) can therefore be further simplified to
\begin{equation}\label{eqs3_16}
\begin{aligned}
&\frac{1}{2}\int_0^t\|T_s(x,s)\|^2_{L^2(\Omega_2)}ds+\frac{a}{2}\|T_t(x,t)\|_{L^2(\Omega_2)}^2+\frac{K}{2}\|T_x(x,t)\|_{L^2(\Omega_2)}^2+Kb\int^t_0\|T_{sx}(x,s)\|^2_{L^2(\Omega_2)}ds\\
\leq&\frac{K^{\frac{3}{2}}(2N+1)}{4\sqrt{z_0}}\int^t_0Q^2(s)ds+\frac{\sqrt{K}z_0^{\frac{3}{2}}}{4(a+b)}\int_0^t\sum_{n=1}^N \frac{a_n}{b^2_n} R^2(s)ds+\frac{K}{2}\|(\xi_1)_x(x)\|^2_{L^2(\Omega_2)}\\
&+\frac{1}{2}\int_0^t\|F(x,s)\|^2ds+\frac{a}{2}\|\eta_1(x)\|^2_{L^2(\Omega_2)}.
\end{aligned}
\end{equation}
From $T(0,t)=0$, for any $x\in[0,x_r]$, we have
\begin{equation}
    \begin{aligned}
        |T(x,t)|^2&\leq\left(\int_0^x1^2dy\right)\left(\int_0^x\left|\frac{\partial T}{\partial y}(y,t)\right|^2dy\right)\leq x\|T_x(x,t)\|^2_{L^2(\Omega_2)},\nonumber
    \end{aligned}
\end{equation}
i.e.,
\begin{equation}\label{eqs3_17}
\begin{aligned}
\|T(x,t)\|_{L^2(\Omega_2)}^2\leq\frac{x_r^2}{2}\|T_x(x,t)\|_{L^2(\Omega_2)}^2.
\end{aligned}
\end{equation}
Thus, applying (\ref{eqs3_17}) to (\ref{eqs3_16}) immediately yields
\begin{equation}
\begin{aligned}
\|T(x,t)\|^2_{L^2(\Omega_2)}
\leq&\frac{x_r^2K^{\frac{1}{2}}(2N+1)}{4\sqrt{z_0}}\int^t_0Q^2(s)ds+\frac{x_r^2z_0^{\frac{3}{2}}}{4\sqrt{K}(a+b)}\int_0^t\sum_{n=1}^N \frac{a_n}{b^2_n} R^2(s)ds+\frac{x_r^2}{2}\|(\xi_1)_x(x)\|^2_{L^2(\Omega_2)}
\nonumber\\
&+\frac{x_r^2}{2K}\int_0^t\|F(x,s)\|^2_{L^2(\Omega_2)}ds+\frac{ax_r^2}{2K}\|\eta_1(x)\|^2_{L^2(\Omega_2)}.\nonumber
\end{aligned}
\end{equation}
\end{proof}

\section{Numerical scheme for the problem (\ref{eqn-250925A})}\label{Sec3}

In this section, we establish numerical scheme for the reduced problem  (\ref{eqn-250925A}), and discuss its stability and error estimate. Throughout the following discussion, $C$ is assumed to be a positive constant, but its value may differ in different situations.
Let $M_t$ and $M_s$ be two positive integers which are used to divide the domains $[0,D]$ and $\Omega_2$ with temporal step-size $\Delta t=\frac{D}{M_t}$ and spatial
step-size $h=\frac{X_R}{M_s},$ respectively.
Define $t_{k-\frac{1}{2}}=\frac12(t_k+t_{k+1}),x_{k-\frac{1}{2}}=\frac12(x_k+x_{k+1})$ and $\Omega_{st}=\Omega_s\times \Omega_{t}$, where $t_k=k\Delta t,$ $x_k=kh,$ $\Omega_t=\{t_k|t_k=k\Delta t,0\leq k \leq M_t\},$ and $\Omega_s=\{x_k|x_k=kh,0\leq k\leq M_s\}.$
For a grid function $v=\{v_i^k|0\leq i\leq M_s,0\leq k \leq M_t\}$ defined on $\Omega_{st}$, denote
$$\delta_tv_i^{k-\frac{1}{2}}=\frac{v_i^k-v_i^{k-1}}{\Delta t},v_i^{k-\frac{1}{2}}=\frac{1}{2}(v_i^k+v_i^{k-1}),0\leq i\leq M_s,1\leq k \leq M_t,$$
$$\delta_x\delta_tv^{k-\frac{1}{2}}_{i-\frac{1}{2}}=\frac{\delta_tv^{k-\frac{1}{2}}_{i}-\delta_tv^{k-\frac{1}{2}}_{i-1}}{h},1\leq i\leq M_s,1\leq k \leq M_t,$$
and
$$\delta_xv^k_{i-\frac{1}{2}}=\frac{v^k_{i}-v^k_{i-1}}{h},\delta^2_xv^k_{i}=\frac{\delta_xv^k_{i+\frac{1}{2}}-\delta_xv^k_{i-\frac{1}{2}}}{h},1\leq i\leq M_s-1,0\leq k \leq M_t.$$

\subsection{Construction of numerical scheme}
In order to discretize (\ref{eqn-250925A}), we first transform the boundary condition (\ref{eqs2_2.10e}) by the following lemma.

\begin{lem}\cite{Gao2013finite}\label{Lemma1}
If $f(x)\in C^3[0,x_r],$ then
\begin{equation}
\begin{aligned}
&f''(x_{M_s})=\frac{2}{h}[f'(x_{M_s})-\delta_xf_{M_s-\frac{1}{2}}]-\frac{h}{3}f^{(3)}(x_{M_s}+\theta h), \theta\in(0,1).\nonumber
\end{aligned}
\end{equation}
\end{lem}
It follows from Lemma \ref{Lemma1} and (\ref{eqs2_2.10a}) that
\begin{equation}\label{eqs3_2710c}
    \begin{aligned}
T_t(x_{M_s},t)+aT_{tt}(x_{M_s},t)=&K\left[\frac{2}{h}\left((T_x(x_{M_s},t)+bT_{tx}(x_{M_s},t)-\delta_xT_{M_s-\frac{1}{2}}(t)-b\delta_x(T_t)_{M_s-\frac{1}{2}}\right)\right.\\
&\left.-\frac{h}{3}\frac{\partial^3(T+bT_t)}{\partial x^3}(x_{M_s}+\theta h)\right]+F(x_{M_s},t).
    \end{aligned}
\end{equation}
Due to (\ref{eqs2_2.10e}) implies that
\begin{equation}\label{eqs3_2710d}
    \begin{aligned}
T_x(x_{M_s},t)+bT_{tx}(x_{M_s},t)=-\sqrt{\frac{z_0}{K}}\left[\left(1+\sum_{n=1}^N \frac{a_n}{b_n}\right) T_t(x_{M_s},t)-z_0 \sum_{n=1}^N \frac{a_n}{b_n} \zeta^{'}_n( t)\right]+Q(t),
    \end{aligned}
\end{equation}
(\ref{eqs3_2710c}) leads to
\begin{equation}\label{eqs3_2710e}
    \begin{aligned}
T_t(x_{M_s},t)+aT_{tt}(x_{M_s},t)
=&\frac{2K}{h}\Bigg\{-\sqrt{\frac{z_0}{K}}\Bigg[\Bigg(1+\sum_{n=1}^N \frac{a_n}{b_n}\Bigg) T_t(x_{M_s},t)-z_0 \sum_{n=1}^N \frac{a_n}{b_n} \zeta^{'}_n( t)\Bigg]+Q(t)\\
&-\delta_xT_{M_s-\frac{1}{2}}-b\delta_x(T_t)_{M_s-\frac{1}{2}}\Bigg\}-\frac{Kh}{3}\frac{\partial^3(T+bT_t)}{\partial x^3}(x_{M_s}+\theta h)+F(x_{M_s},t).
    \end{aligned}
\end{equation}
Now for (\ref{eqn-250925A}), we replace (\ref{eqs2_2.10e}) with (\ref{eqs3_2710e}), introduce new variables
$$w=T_t,\varsigma_n=\zeta^{'}_n, n=1,\dots,N,$$
employ Crank-Nicolson method in time direction and use central difference formula to discretize spatial derivatives. Then it yields
\begin{equation}
\left\{ \begin{aligned}
&w^{k-\frac{1}{2}}_i+a\delta_tw^{k-\frac{1}{2}}_i=K\left(\delta^2_xT^{k-\frac{1}{2}}_i+b\delta^2_xw^{k-\frac{1}{2}}_i\right)+F^{k-\frac{1}{2}}_i+P^{k-\frac{1}{2}}_i,1\leq i \leq M_s-1,1\leq k\leq M_t,\\
&T^0_i=(\xi_1)_i,0\leq i \leq M_s,\\
&w^0_i=(\eta_1)_i,0\leq i \leq M_s,\\
&T^{k-\frac{1}{2}}_0=0,1\leq k\leq M_t,\\
&w^{k-\frac{1}{2}}_{M_s}+a\delta_tw^{k-\frac{1}{2}}_{M_s}=\frac{2K}{h}\Bigg\{-\delta_xT^{k-\frac{1}{2}}_{{M_s}-\frac{1}{2}}-b\delta_{x}w^{k-\frac{1}{2}}_{{M_s}-\frac{1}{2}}-\sqrt{\frac{z_0}{K}}\Bigg[\left(1+\sum_{n=1}^N \frac{a_n}{b_n}\right) w^{k-\frac{1}{2}}_{M_s}\\
&-z_0 \sum_{n=1}^N \frac{a_n}{b_n} \varsigma^{k-\frac{1}{2}}_n\Bigg]+Q^{k-\frac{1}{2}}\Bigg\}+F^{k-\frac{1}{2}}_{M_s}
+Q^{k-\frac{1}{2}}_{M_s},1\leq k\leq M_t , \\
&(z_0-b_n z_0)\varsigma^{k-\frac{1}{2}}_n+b_n\left[\zeta^{k-\frac{1}{2}}_n+(a+b)\varsigma^{k-\frac{1}{2}}_n+ab\delta_t\varsigma^{k-\frac{1}{2}}_n\right]=w^{k-\frac{1}{2}}_{M_s}+R^{k-\frac{1}{2}}+R^{k-\frac{1}{2}}_{M_s},1\leq k\leq M_t, \\
&\zeta^0_n=\varsigma^0_n=0,n=1,\dots,N,\\
&\delta_xw^{k-\frac{1}{2}}_{i-\frac{1}{2}}-\delta_x\delta_tT^{k-\frac{1}{2}}_{i-\frac{1}{2}}=G^{k-\frac{1}{2}}_{i-\frac{1}{2}},1\leq i \leq M_s,1\leq k\leq M_t,\\
&\varsigma_n^{k-\frac{1}{2}}-\delta_{t}\zeta_n^{k-\frac{1}{2}}=G^{k-\frac{1}{2}}_1,1\leq k\leq M_t,
\end{aligned}\right.\label{eqs4_1}
\end{equation}
where these local truncation errors satisfy
\begin{equation}\label{eqs4_2}
\begin{aligned}
\bigg |P_i^{k-\frac{1}{2}}\bigg|\leq C(\Delta t^2+h^2),1\leq k\leq M_t,1\leq i \leq M_s-1,
\end{aligned}
\end{equation}
\begin{equation}\label{eqs4_4}
\begin{aligned}
\bigg|G_{i-\frac{1}{2}}^{k-\frac{1}{2}}\bigg|\leq C(\Delta t^2+h^2),1\leq k\leq M_t,1\leq i \leq M_s,
\end{aligned}
\end{equation}
\begin{equation}\label{eqs4_3}
\begin{aligned}
\bigg|Q^{k-\frac{1}{2}}_{M_s}\bigg |\leq C(\Delta t^2+h),
\bigg|R^{k-\frac{1}{2}}_{M_s}\bigg|\leq C\Delta t^2,\bigg|G_{1}^{k-\frac{1}{2}}\bigg|\leq C\Delta t^2,1\leq k\leq M_t.
\end{aligned}
\end{equation}
Omitting these local truncation errors in (\ref{eqs4_1}) yields the following finite difference scheme of the problem (\ref{eqn-250925A})
\begin{subequations}\label{eqn4_5}
\begin{numcases}{}
\bar{W}^{k-\frac{1}{2}}_i+a\delta_t\bar{W}^{k-\frac{1}{2}}_i=K\left(\delta^2_x\bar{T}^{k-\frac{1}{2}}_{i}+b\delta^2_x\bar{W}^{k-\frac{1}{2}}_i\right)+F^{k-\frac{1}{2}}_i,1\leq i \leq M_s-1,1\leq k\leq M_t,\label{eqs4_4.5a}\\
\bar{T}^0_{i}=(\xi_1)_i,0\leq i \leq M_s,\label{eqs4_4.5b}\\
\bar{W}^0_i=(\eta_1)_i,0\leq i \leq M_s,\label{eqs4_4.5c}\\
\bar{T}^{k-\frac{1}{2}}_{0}=0,1\leq k\leq M_t,\label{eqs4_4.5d}\\
\bar{W}^{k-\frac{1}{2}}_{M_s}+a\delta_t\bar{W}^{k-\frac{1}{2}}_{M_s}=\frac{2K}{h}\Bigg\{-\delta_x\bar{T}^{k-\frac{1}{2}}_{{M_s}-\frac{1}{2}}-b\delta_{x}\bar{W}^{k-\frac{1}{2}}_{{M_s}-\frac{1}{2}}-\sqrt{\frac{z_0}{K}}\Bigg[\left(1+\sum_{n=1}^N \frac{a_n}{b_n}\right) \bar{W}^{k-\frac{1}{2}}_{M_s}\nonumber\\
-z_0 \sum_{n=1}^N \frac{a_n}{b_n} \sigma^{k-\frac{1}{2}}_n\Bigg]+Q^{k-\frac{1}{2}}\Bigg\}+F^{k-\frac{1}{2}}_{M_s}
,1\leq k\leq M_t, \label{eqs4_4.5e}\\
(z_0-b_n z_0)\sigma^{k-\frac{1}{2}}_n+b_n\left[Z^{k-\frac{1}{2}}_n+(a+b)\sigma^{k-\frac{1}{2}}_n+ab\delta_t\sigma^{k-\frac{1}{2}}_n\right]=\bar{W}^{k-\frac{1}{2}}_{M_s}+R^{k-\frac{1}{2}},1\leq k\leq M_t,\label{eqs4_4.5f}\\
Z^0_n=\sigma^0_n=0,n=1,\dots,N,\label{eqs4_4.5g}\\
\delta_x\bar{W}^{k-\frac{1}{2}}_{i-\frac{1}{2}}-\delta_x\delta_t\bar{T}^{k-\frac{1}{2}}_{i-\frac{1}{2}}=0,1\leq i \leq M_s,1\leq k\leq M_t,\label{eqs4_4.5h}\\
\sigma_n^{k-\frac{1}{2}}-\delta_{t}Z_n^{k-\frac{1}{2}}=0,1\leq k\leq M_t.\label{eqs4_4.5i}
\end{numcases}
\end{subequations}

\subsection{Stability and convergence of numerical scheme}
In this subsection, we discuss the stability and convergence of the finite difference scheme (\ref{eqn4_5}).
Let
$$\begin{gathered}
V=\{v_i|0\leq i\leq M_s~\text{and}~v_{0}=0\}.
\end{gathered}$$
For any $u,v\in V$, define
$$
\begin{aligned}
&\langle u,v\rangle= h\sum_{i=1}^{M_{s}-1}u_{i}v_{i}+\frac{h}{2}\sum_{i=0,M_{s}}u_{i}v_{i}, \quad\|u\|=\langle u,u\rangle^{1/2}, \quad\|u\|_1^2=h\sum_{i=1}^{M_{s}}\left(\delta_xu_{i-\frac{1}{2}}\right)^2,\|u\|_{\infty}=\max_{0\leq i \leq M_s}|u_i|. \\
\end{aligned}
$$
We introduce two using lemmas as follows.
\begin{lem}\label{Lemma2}
\cite{Zhao2021Efficient}
Let $\{X_n\}$ and $\left\{Y_n\right\}$ be two non-negative sequences, and the sequence  $\{\phi_n\}$ satisfies
$$
\phi_0\leq G_0,\quad\phi_n\leq G_0+\sum_{k=0}^{n-1}X_k+\sum_{k=0}^{n-1}Y_k\phi_k,\quad n\geq1,
$$
where $G_0\geq0.$ Then the sequence $\{\phi_n\}$ satisfies
$$
\phi_n\leq\left(G_0+\sum_{k=0}^{n-1}X_k\right)\exp\left(\sum_{k=0}^{n-1}Y_k\right),\quad n\geq1.
$$
\end{lem}
\begin{lem}\label{Lemma3}
\cite{Sun2005Numerical}
For any grid function $u=\{u^k_{i}|0\leq i \leq M_s,0\leq k \leq M_t\ \text{and}\ u^k_0=0\}$ defined on $\Omega_{st}$, one has
$$\|u^m\|_{\infty}\leq\|u^m\|_1, m=0,1,\cdots,M_t.$$
\end{lem}

Now we discuss the following lemma.

\begin{lem}\label{Lemma4}
Assume that $\{\bar{T}^k_{i}|0\leq i\leq M_s,0\leq k\leq M_t\}$ is the solution of the following numerical scheme
\begin{subequations}\label{eqn4_6}
\begin{numcases}{}
\bar{W}^{k-\frac{1}{2}}_i+a\delta_t\bar{W}^{k-\frac{1}{2}}_i=K\left(\delta^2_x\bar{T}^{k-\frac{1}{2}}_{i}+b\delta^2_x\bar{W}^{k-\frac{1}{2}}_i\right)+g^{k-\frac{1}{2}}_i,1\leq i \leq M_s-1,1\leq k\leq M_t,\label{eqs4_4.6a}\\
\bar{T}^0_{i}=(\xi_1)_i,0\leq i \leq M_s,\label{eqs4_4.6b}\\
\bar{W}^0_i=(\eta_1)_i,0\leq i \leq M_s,\label{eqs4_4.6c}\\
\bar{T}^{k-\frac{1}{2}}_{0}=0,1\leq k\leq M_t,\label{eqs4_4.6d}\\
\bar{W}^{k-\frac{1}{2}}_{M_s}+a\delta_t\bar{W}^{k-\frac{1}{2}}_{M_s}=\frac{2K}{h}\Bigg\{-\delta_x\bar{T}^{k-\frac{1}{2}}_{{M_s}-\frac{1}{2}}-b\delta_{x}\bar{W}^{k-\frac{1}{2}}_{{M_s}-\frac{1}{2}}-\sqrt{\frac{z_0}{K}}\Bigg[\left(1+\sum_{n=1}^N \frac{a_n}{b_n}\right) \bar{W}^{k-\frac{1}{2}}_{M_s}\nonumber\\-z_0 \sum_{n=1}^N \frac{a_n}{b_n} \sigma^{k-\frac{1}{2}}_n\Bigg]+Q^{k-\frac{1}{2}}\Bigg\}+g^{k-\frac{1}{2}}_{M_s}
,1\leq k\leq M_t,\label{eqs4_4.6e} \\
(z_0-b_n z_0)\sigma^{k-\frac{1}{2}}_n+b_n\left[Z^{k-\frac{1}{2}}_n+(a+b)\sigma^{k-\frac{1}{2}}_n+ab\delta_t\sigma^{k-\frac{1}{2}}_n\right]=\bar{W}^{k-\frac{1}{2}}_{M_s}+r^{k-\frac{1}{2}},1\leq k\leq M_t,\label{eqs4_4.6f}\\
Z^0_n=\sigma^0_n=0,n=1,\dots,N,\label{eqs4_4.6g}\\
\delta_x\bar{W}^{k-\frac{1}{2}}_{i-\frac{1}{2}}-\delta_x\delta_t\bar{T}^{k-\frac{1}{2}}_{i-\frac{1}{2}}=G^{k-\frac{1}{2}}_{i-\frac{1}{2}},1\leq i \leq M_s,1\leq k\leq M_t,\label{eqs4_4.6h}\\
\sigma_n^{k-\frac{1}{2}}-\delta_{t}Z_n^{k-\frac{1}{2}}=G^{k-\frac{1}{2}}_1,1\leq k\leq M_t.\label{eqs4_4.6i}
\end{numcases}
\end{subequations}
Then for $1\leq m\leq M_t$, we have
\begin{equation}\label{eqs3_30}
\begin{aligned}
\|\bar{T}^{m}\|_{\infty}^2
\leq&\Bigg\{
\frac{K}{2}\|\xi_1\|_1^2
+\frac{\Delta t}{2}\sum^m_{k=1}\|g^{k-\frac{1}{2}}\|^2+\frac{(2N+1)}{8\sqrt{K}\sqrt{z_0}}\Delta t\sum^m_{k=1}(hg_{M_s}^{k-\frac{1}{2}})^2+K^{\frac{3}{2}}\Delta t\sum^m_{k=1} \frac{2N+1}{2\sqrt{z_0}}(Q^{k-\frac{1}{2}})^2
\\
&+\frac{\sqrt{K}}{4(a+b)}\Delta t\sum^m_{k=1}z_0^{\frac{3}{2}} \sum_{n=1}^N a_n(\frac{1}{b_n}r^{k-\frac{1}{2}})^2+\frac{K}{2}\Delta t\sum^m_{k=1}\|G^{k-\frac{1}{2}}\|^2+ \frac{\sqrt{K}z_0^{\frac{3}{2}}}{2}\sum_{n=1}^N a_n\Delta t\sum^m_{k=1} (G^{k-\frac{1}{2}}_1)^2\\
&+\frac{a}{2}\|\eta_1\|^2\Bigg\}\frac{2e^D}{K}.
\end{aligned}
\end{equation}
\end{lem}
\begin{proof}
Multiplying both sides of (\ref{eqs4_4.6a}) by $h\bar{W}^{k-\frac{1}{2}}_{i}$ and summing from
$i=1$ to $M_s-1$, and multiplying both sides of (\ref{eqs4_4.6e}) by $\frac{h}{2}\bar{W}^{k-\frac{1}{2}}_{M_s}$ yields
\begin{equation}\label{eqs4_7}
\begin{aligned}
&h\sum^{M_s-1}_{i=1}(\bar{W}^{k-\frac{1}{2}}_{i})^2+ah\sum^{M_s-1}_{i=1}\bar{W}^{k-\frac{1}{2}}_{i}\delta_t\bar{W}^{k-\frac{1}{2}}_{i}\\
=&Kh\sum^{M_s-1}_{i=1}\bar{W}^{k-\frac{1}{2}}_{i}\delta_{xx}\bar{T}^{k-\frac{1}{2}}_{i}+Kbh\sum^{M_s-1}_{i=1}\bar{W}^{k-\frac{1}{2}}_{i}\delta_{xx}\bar{W}^{k-\frac{1}{2}}_i+h\sum^{M_s-1}_{i=1}\bar{W}^{k-\frac{1}{2}}_{i}g^{k-\frac{1}{2}}_i,
\end{aligned}
\end{equation}
and
\begin{equation}\label{eqs4_8}
\begin{aligned}
&\frac{h}{2}(\bar{W}^{k-\frac{1}{2}}_{M_s})^2+a\frac{h}{2}\bar{W}^{k-\frac{1}{2}}_{M_s}\delta_t\bar{W}^{k-\frac{1}{2}}_{M_s}\\
=&K\Bigg\{-\delta_x\bar{T}^{k-\frac{1}{2}}_{{M_s}-\frac{1}{2}}-b\delta_{x}\bar{W}^{k-\frac{1}{2}}_{{M_s}-\frac{1}{2}}-\sqrt{\frac{z_0}{K}}\Bigg[\left(1+\sum_{n=1}^N \frac{a_n}{b_n}\right) \bar{W}^{k-\frac{1}{2}}_{M_s}
-z_0 \sum_{n=1}^N \frac{a_n}{b_n} \sigma^{k-\frac{1}{2}}_n\Bigg]+Q^{k-\frac{1}{2}}\Bigg\}\bar{W}^{k-\frac{1}{2}}_{M_s}\\
&+\frac{h}{2}\bar{W}^{k-\frac{1}{2}}_{M_s}g^{k-\frac{1}{2}}_{M_s}.
\end{aligned}
\end{equation}
By summing (\ref{eqs4_7}) and (\ref{eqs4_8}), and taking (\ref{eqs4_4.6d}) into account, we obtain
\begin{equation}\label{eqs4_9}
\begin{aligned}
&\|\bar{W}^{k-\frac{1}{2}}\|^2+a\langle \bar{W}^{k-\frac{1}{2}},\delta_t\bar{W}^{k-\frac{1}{2}}\rangle\\
=&K(h\sum^{M_s-1}_{i=1}\bar{W}^{k-\frac{1}{2}}_{i}\delta_{xx}\bar{T}^{k-\frac{1}{2}}_{i}-\delta_x\bar{T}^{k-\frac{1}{2}}_{{M_s}-\frac{1}{2}}\bar{W}^{k-\frac{1}{2}}_{M_s}+\delta_x\bar{T}^{k-\frac{1}{2}}_{\frac{1}{2}}\bar{W}^{k-\frac{1}{2}}_0)\\
&+Kb(h\sum^{M_s-1}_{i=1}\bar{W}^{k-\frac{1}{2}}_{i}\delta_{xx}\bar{W}^{k-\frac{1}{2}}_i-\delta_{x}\bar{W}^{k-\frac{1}{2}}_{{M_s}-\frac{1}{2}}\bar{W}^{k-\frac{1}{2}}_{M_s}+\delta_{x}\bar{W}^{k-\frac{1}{2}}_{\frac{1}{2}}\bar{W}^{k-\frac{1}{2}}_0)\\
&-\sqrt{z_0K}\left(1+\sum_{n=1}^N \frac{a_n}{b_n}\right) (\bar{W}^{k-\frac{1}{2}}_{M_s})^2
+\sqrt{K}z_0^{\frac{3}{2}} \sum_{n=1}^N \frac{a_n}{b_n} \sigma^{k-\frac{1}{2}}_n\bar{W}^{k-\frac{1}{2}}_{M_s}+\langle W^{k-\frac{1}{2}},g^{k-\frac{1}{2}}\rangle+KQ^{k-\frac{1}{2}}\bar{W}^{k-\frac{1}{2}}_{M_s}.
\end{aligned}
\end{equation}
Since
\begin{equation}
\begin{aligned}
&h\sum^{M_s-1}_{i=1}\bar{W}^{k-\frac{1}{2}}_{i}\delta_{xx}\bar{T}^{k-\frac{1}{2}}_{i}-\delta_x\bar{T}^{k-\frac{1}{2}}_{{M_s}-\frac{1}{2}}\bar{W}^{k-\frac{1}{2}}_{M_s}+\delta_x\bar{T}^{k-\frac{1}{2}}_{\frac{1}{2}}\bar{W}^{k-\frac{1}{2}}_0=-h\sum^{M_s}_{i=1}\delta_x\bar{T}^{k-\frac{1}{2}}_{i-\frac{1}{2}}\delta_{x}\bar{W}^{k-\frac{1}{2}}_{i-\frac{1}{2}},\nonumber
\end{aligned}
\end{equation}
and
\begin{equation}
\begin{aligned}
&h\sum^{M_s-1}_{i=1}\bar{W}^{k-\frac{1}{2}}_{i}\delta_{xx}\bar{W}^{k-\frac{1}{2}}_i-\delta_{x}\bar{W}^{k-\frac{1}{2}}_{{M_s}-\frac{1}{2}}\bar{W}^{k-\frac{1}{2}}_{M_s}+\delta_{x}\bar{W}^{k-\frac{1}{2}}_{\frac{1}{2}}\bar{W}^{k-\frac{1}{2}}_0=-\|\bar{W}^{k-\frac{1}{2}}\|_1^2,\nonumber
\end{aligned}
\end{equation}
(\ref{eqs4_9}) can be reformulated as
\begin{equation}\label{eqs4_10}
\begin{aligned}
&\|\bar{W}^{k-\frac{1}{2}}\|^2+a\langle \bar{W}^{k-\frac{1}{2}},\delta_t\bar{W}^{k-\frac{1}{2}}\rangle+\sqrt{K}z_0^{\frac{3}{2}} \sum_{n=1}^N \frac{a_n}{b_n} \sigma^{k-\frac{1}{2}}_n\bar{W}^{k-\frac{1}{2}}_{M_s}\\
=&-Kh\sum^{M_s}_{i=1}\delta_x\bar{T}^{k-\frac{1}{2}}_{i-\frac{1}{2}}\delta_{x}\bar{W}^{k-\frac{1}{2}}_{i-\frac{1}{2}}-Kb\|\bar{W}^{k-\frac{1}{2}}\|_1^2-\sqrt{z_0K}\left(1+\sum_{n=1}^N \frac{a_n}{b_n}\right) (\bar{W}^{k-\frac{1}{2}}_{M_s})^2
+2\sqrt{K}z_0^{\frac{3}{2}} \sum_{n=1}^N \frac{a_n}{b_n} \sigma^{k-\frac{1}{2}}_n\bar{W}^{k-\frac{1}{2}}_{M_s}\\
&+\langle \bar{W}^{k-\frac{1}{2}},g^{k-\frac{1}{2}}\rangle+KQ^{k-\frac{1}{2}}\bar{W}^{k-\frac{1}{2}}_{M_s}.
\end{aligned}
\end{equation}
According to (\ref{eqs4_4.6f}), the replacement of $\bar{W}^{k-\frac{1}{2}}_{M_s}$ in the third term on the left hand side of (\ref{eqs4_10}) leads to
\begin{equation}\label{eqs4_11}
\begin{aligned}
&\|\bar{W}^{k-\frac{1}{2}}\|^2+a\langle \bar{W}^{k-\frac{1}{2}},\delta_t\bar{W}^{k-\frac{1}{2}}\rangle+\sqrt{K}z_0^{\frac{3}{2}} \sum_{n=1}^N a_n \sigma^{k-\frac{1}{2}}_n\left[Z^{k-\frac{1}{2}}_n+(a+b)\sigma^{k-\frac{1}{2}}_n+ab\delta_t\sigma^{k-\frac{1}{2}}_n\right]\\
=&-Kh\sum^{M_s}_{i=1}\delta_x\bar{T}^{k-\frac{1}{2}}_{i-\frac{1}{2}}\delta_{x}\bar{W}^{k-\frac{1}{2}}_{i-\frac{1}{2}}-Kb\|\bar{W}^{k-\frac{1}{2}}\|_1^2\\
&-\sqrt{K}\Bigg[\sqrt{z_0}\left(1+\sum_{n=1}^N \frac{a_n}{b_n}\right) (\bar{W}^{k-\frac{1}{2}}_{M_s})^2
-2z_0^{\frac{3}{2}} \sum_{n=1}^N \frac{a_n}{b_n} \sigma^{k-\frac{1}{2}}_n\bar{W}^{k-\frac{1}{2}}_{M_s}+z_0^{\frac{5}{2}} \sum_{n=1}^N \frac{a_n(1-b_n)}{b_n}(\sigma^{k-\frac{1}{2}}_n)^2\Bigg]\\
&+\langle W^{k-\frac{1}{2}},g^{k-\frac{1}{2}}\rangle
+\sqrt{K}z_0^{\frac{3}{2}} \sum_{n=1}^N \frac{a_n}{b_n} \sigma^{k-\frac{1}{2}}_nr^{k-\frac{1}{2}}+KQ^{k-\frac{1}{2}}\bar{W}^{k-\frac{1}{2}}_{M_s}.
\end{aligned}
\end{equation}
Similar to (\ref{eqs3_9}), the third term on the right hand side of (\ref{eqs4_11}) can be transformed into the following form
\begin{equation}
\begin{aligned}
&-\sqrt{z_0}(1+\sum_{n=1}^N \frac{a_n}{b_n})(\bar{W}^{k-\frac{1}{2}}_{M_s})^2+2z_0^{\frac{3}{2}} \sum_{n=1}^N \frac{a_n}{b_n} \sigma_n^{k-\frac{1}{2}}\bar{W}^{k-\frac{1}{2}}_{M_s}-z_0^{\frac{5}{2}} \sum_{n=1}^N \frac{a_n(1-b_n)}{b_n}(\sigma
^{k-\frac{1}{2}}_n)^2\nonumber\\
=&-\sum_{n=1}^N \frac{2\sqrt{z_0}}{(2N+1)b_n}\left[(1-b_n)z_0\sigma
^{k-\frac{1}{2}}_n-\bar{W}^{k-\frac{1}{2}}_{M_s} \right]^2- \frac{\sqrt{z_0}}{2N+1}(\bar{W}^{k-\frac{1}{2}}_{M_s})^2,\nonumber
\end{aligned}
\end{equation}
we can therefore obtain
\begin{equation}\label{eqs4_12}
\begin{aligned}
&\|\bar{W}^{k-\frac{1}{2}}\|^2+a\langle \bar{W}^{k-\frac{1}{2}},\delta_t\bar{W}^{k-\frac{1}{2}}\rangle+\sqrt{K}z_0^{\frac{3}{2}} \sum_{n=1}^N a_n \sigma^{k-\frac{1}{2}}_nZ^{k-\frac{1}{2}}_n+\sqrt{K}z_0^{\frac{3}{2}} \sum_{n=1}^N a_n \sigma^{k-\frac{1}{2}}_n\left[(a+b)\sigma^{k-\frac{1}{2}}_n+ab\delta_t\sigma^{k-\frac{1}{2}}_n\right]\\
\leq&-Kh\sum^{M_s}_{i=1}\delta_x\bar{T}^{k-\frac{1}{2}}_{i-\frac{1}{2}}\delta_{x}\bar{W}^{k-\frac{1}{2}}_{i-\frac{1}{2}}-Kb\|\bar{W}^{k-\frac{1}{2}}\|_1^2
-\sqrt{K}\frac{\sqrt{z_0}}{2N+1}(\bar{W}^{k-\frac{1}{2}}_{M_s})^2+\sqrt{K}z_0^{\frac{3}{2}} \sum_{n=1}^N \frac{a_n}{b_n} \sigma^{k-\frac{1}{2}}_nr^{k-\frac{1}{2}}+\langle \bar{W}^{k-\frac{1}{2}},g^{k-\frac{1}{2}}\rangle\\
&+KQ^{k-\frac{1}{2}}\bar{W}^{k-\frac{1}{2}}_{M_s}.
\end{aligned}
\end{equation}
Multiplying both sides of (\ref{eqs4_12}) by $\Delta t$ and summing from $k=1$ to $m$, and using (\ref{eqs4_4.6i}) to replace $\sigma^{k-\frac{1}{2}}_n$ in the third term on the left hand side of (\ref{eqs4_12}), it yields
\begin{equation}\label{eqs4_13}
\begin{aligned}
&\Delta t\sum^m_{k=1}\|\bar{W}^{k-\frac{1}{2}}\|^2+a\Delta t\sum^m_{k=1}\langle \bar{W}^{k-\frac{1}{2}},\delta_t\bar{W}^{k-\frac{1}{2}}\rangle+Kb\Delta t\sum^m_{k=1}\|\bar{W}^{k-\frac{1}{2}}\|_1^2\\
&+\sqrt{K}z_0^{\frac{3}{2}} \sum_{n=1}^N a_n \left[\frac{1}{2}(Z_n^m)^2+\Delta t\sum^m_{k=1} G^{k-\frac{1}{2}}_1Z^{k-\frac{1}{2}}_n+(a+b)\Delta t\sum^m_{k=1} (\sigma^{k-\frac{1}{2}}_n)^2+\frac{ab}{2}(\sigma_n^m)^2\right]\\
\leq&-Kh\sum^{M_s}_{i=1}\Delta t\sum^m_{k=1}\delta_x\bar{T}^{k-\frac{1}{2}}_{i-\frac{1}{2}}\delta_{x}\bar{W}^{k-\frac{1}{2}}_{i-\frac{1}{2}}
+\Delta t\sum^m_{k=1}\langle \bar{W}^{k-\frac{1}{2}},g^{k-\frac{1}{2}}\rangle-\sqrt{K}\Delta t\sum^m_{k=1}\frac{\sqrt{z_0}}{2N+1}(\bar{W}^{k-\frac{1}{2}}_{M_s})^2
\\
&+\sqrt{K}z_0^{\frac{3}{2}} \sum_{n=1}^N \frac{a_n}{b_n}\Delta t\sum^m_{k=1} \sigma^{k-\frac{1}{2}}_nr^{k-\frac{1}{2}}+K\Delta t\sum^m_{k=1} Q^{k-\frac{1}{2}}\bar{W}^{k-\frac{1}{2}}_{M_s}.
\end{aligned}
\end{equation}
Noting that
\begin{equation}\label{eqS4_15}
\begin{aligned}
a\Delta t\sum^m_{k=1}\langle \bar{W}^{k-\frac{1}{2}},\delta_{t}\bar{W}^{k-\frac{1}{2}}\rangle
&=\frac{a}{2}\|\bar{W}^{m}\|^2-\frac{a}{2}\|\eta_1\|^2,
\end{aligned}
\end{equation}
\begin{equation}\label{eqs4_16}
\begin{aligned}
\langle \bar{W}^{k-\frac{1}{2}},g^{k-\frac{1}{2}}\rangle=h\sum^{M_s-1}_{i=1}\bar{W}^{k-\frac{1}{2}}_{i}g^{k-\frac{1}{2}}_i+\frac{h}{2}\bar{W}^{k-\frac{1}{2}}_{M_s}g^{k-\frac{1}{2}}_{M_s},
\end{aligned}
\end{equation}
and it follows from (\ref{eqs4_4.6h}) that
\begin{equation}\label{eqs4_14}
\begin{aligned}
h\sum^{M_s}_{i=1}\Delta t\sum^m_{k=1}\delta_x\bar{T}^{k-\frac{1}{2}}_{i-\frac{1}{2}}\delta_{x}\bar{W}^{k-\frac{1}{2}}_{i-\frac{1}{2}}&=h\Delta t\sum^m_{k=1}\sum^{M_s}_{i=1}\delta_x\bar{T}^{k-\frac{1}{2}}_{i-\frac{1}{2}}\delta_{x}\delta_t\bar{T}^{k-\frac{1}{2}}_{i-\frac{1}{2}}+h\Delta t\sum^m_{k=1}\sum^{M_s}_{i=1}\delta_x\bar{T}^{k-\frac{1}{2}}_{i-\frac{1}{2}}G^{k-\frac{1}{2}}\\
&=\frac{1}{2}h\sum^m_{k=1}\sum^{M_s}_{i=1}[(\delta_x\bar{T}^k_{i-\frac{1}{2}})^2-(\delta_x\bar{T}^{k-1}_{i-\frac{1}{2}})^2]+h\Delta t\sum^m_{k=1}\sum^{M_s}_{i=1}\delta_x\bar{T}^{k-\frac{1}{2}}_{i-\frac{1}{2}}G_{i-\frac{1}{2}}^{k-\frac{1}{2}}\\
&=\frac{1}{2}\|\bar{T}^m\|_1^2-\frac{1}{2}\|\xi_1\|_1^2+h\Delta t\sum^m_{k=1}\sum^{M_s}_{i=1}\delta_x\bar{T}^{k-\frac{1}{2}}_{i-\frac{1}{2}}G_{i-\frac{1}{2}}^{k-\frac{1}{2}}.
\end{aligned}
\end{equation}
Therefore,
the substitution of (\ref{eqs4_14})--(\ref{eqs4_16}) into (\ref{eqs4_13}) results in
\begin{equation}\label{eqs4_17}
\begin{aligned}
&\Delta t\sum^m_{k=1}\|\bar{W}^{k-\frac{1}{2}}\|^2+\frac{a}{2}\|\bar{W}^{m}\|^2+Kb\Delta t\sum^m_{k=1}\|\bar{W}^{k-\frac{1}{2}}\|_1^2+\frac{K}{2}\|\bar{T}^m\|_1^2\\
&+\sqrt{K}z_0^{\frac{3}{2}} \sum_{n=1}^N a_n \left[\frac{1}{2}(Z_n^m)^2+\Delta t\sum^m_{k=1} G^{k-\frac{1}{2}}_1Z^{k-\frac{1}{2}}_n+(a+b)\Delta t\sum^m_{k=1} (\sigma^{k-\frac{1}{2}}_n)^2+\frac{ab}{2}(\sigma_n^m)^2\right]\\
\leq&\frac{K}{2}\|\xi_1\|_1^2+\frac{a}{2}\|\eta_1\|^2-Kh\Delta t\sum^m_{k=1}\sum^{M_s}_{i=1}\delta_x\bar{T}^{k-\frac{1}{2}}_{i-\frac{1}{2}}G_{i-\frac{1}{2}}^{k-\frac{1}{2}}
+\Delta t\sum^m_{k=1}h\sum^{M_s-1}_{i=1}\bar{W}^{k-\frac{1}{2}}_{i}g^{k-\frac{1}{2}}_i+\Delta t\sum^m_{k=1}\frac{h}{2}\bar{W}^{k-\frac{1}{2}}_{M_s}g^{k-\frac{1}{2}}_{M_s}\\
&-\sqrt{K}\Delta t\sum^m_{k=1}\frac{\sqrt{z_0}}{2N+1}(\bar{W}^{k-\frac{1}{2}}_{M_s})^2
+\sqrt{K}z_0^{\frac{3}{2}} \sum_{n=1}^N \frac{a_n}{b_n}\Delta t\sum^m_{k=1} \sigma^{k-\frac{1}{2}}_nr^{k-\frac{1}{2}}+K\Delta t\sum^m_{k=1} Q^{k-\frac{1}{2}}\bar{W}^{k-\frac{1}{2}}_{M_s}.
\end{aligned}
\end{equation}

Since the Cauchy-Schwarz inequality implies that
\begin{equation}\label{eqs4_18}
\begin{aligned}
&\Delta t\sum^m_{k=1} G^{k-\frac{1}{2}}_1Z^{k-\frac{1}{2}}_n\leq\frac{1}{2}\Delta t\sum^m_{k=1}(G^{k-\frac{1}{2}}_1)^2+\frac{1}{2}\Delta t\sum^m_{k=1}(Z^{k-\frac{1}{2}}_n)^2,\end{aligned}
\end{equation}
\begin{equation}\label{eqs4_19}
\begin{aligned}
h\Delta t\sum^m_{k=1}\sum^{M_s}_{i=1}\delta_x\bar{T}^{k-\frac{1}{2}}_{i-\frac{1}{2}}G_{i-\frac{1}{2}}^{k-\frac{1}{2}}\leq\frac{1}{2}\Delta t\sum^m_{k=1}\|\bar{T}^{k-\frac{1}{2}}\|_1^2+\frac{1}{2}\Delta t\sum^m_{k=1}\|G^{k-\frac{1}{2}}\|^2,
\end{aligned}
\end{equation}
\begin{equation}\label{eqs4_20}
\begin{aligned}
h\sum^{M_s-1}_{i=1}\bar{W}^{k-\frac{1}{2}}_{i}g^{k-\frac{1}{2}}_i\leq\frac{1}{2}\|\bar{W}^{k-\frac{1}{2}}\|^2+\frac{1}{2}\|g^{k-\frac{1}{2}}\|^2,
\end{aligned}
\end{equation}
\begin{equation}\label{eqs4_21}
\begin{aligned}
\frac{h}{2}\bar{W}^{k-\frac{1}{2}}_{M_s}g^{k-\frac{1}{2}}_{M_s}\leq\frac{\sqrt{K}\sqrt{z_0}}{2(2N+1)}(\bar{W}^{k-\frac{1}{2}}_{M_s})^2+\frac{(2N+1)}{8\sqrt{K}\sqrt{z_0}}(hg^{k-\frac{1}{2}}_{M_s})^2,
\end{aligned}
\end{equation}
\begin{equation}\label{eqs4_22}
\begin{aligned}
\sqrt{K}z_0^{\frac{3}{2}} \sum_{n=1}^N \frac{a_n}{b_n}\Delta t\sum^m_{k=1} \sigma^{k-\frac{1}{2}}_nr^{k-\frac{1}{2}}
\leq\frac{\sqrt{K}}{4(a+b)}\Delta t\sum^m_{k=1}z_0^{\frac{3}{2}} \sum_{n=1}^N a_n(\frac{1}{b_n}r^{k-\frac{1}{2}})^2+\sqrt{K}(a+b)\Delta t\sum^m_{k=1}z_0^{\frac{3}{2}} \sum_{n=1}^N a_n(\sigma^{k-\frac{1}{2}}_n)^2,
\end{aligned}
\end{equation}
and
\begin{equation}\label{eqs4_23}
\begin{aligned}
&K\Delta t\sum^m_{k=1} Q^{k-\frac{1}{2}}\bar{W}^{k-\frac{1}{2}}_{M_s}\leq K^{\frac{3}{2}}\Delta t\sum^m_{k=1} \frac{2N+1}{2\sqrt{z_0}}(Q^{k-\frac{1}{2}})^2+\sqrt{K}\Delta t\sum^m_{k=1}\frac{\sqrt{z_0}}{2(2N+1)}(\bar{W}^{k-\frac{1}{2}}_{M_s})^2,
\end{aligned}
\end{equation}
by substituting (\ref{eqs4_18})--(\ref{eqs4_23}) into (\ref{eqs4_17}), it leads to
\begin{equation}
\begin{aligned}
&\frac{1}{2}\Delta t\sum^m_{k=1}\|\bar{W}^{k-\frac{1}{2}}\|^2+\frac{a}{2}\|\bar{W}^{m}\|^2-\frac{a}{2}\|\eta_1\|^2+Kb\Delta t\sum^m_{k=1}\|\bar{W}^{k-\frac{1}{2}}\|_1^2+\frac{K}{2}\|\bar{T}^{m}\|_1^2\\
&+\sqrt{K}z_0^{\frac{3}{2}} \sum_{n=1}^N a_n \left[\frac{1}{2}(Z_n^m)^2+\frac{ab}{2}(\sigma_n^m)^2\right]\\
\leq&\frac{K}{2}\|\xi_1\|_1^2
+\frac{\Delta t}{2}\sum^m_{k=1}\|g^{k-\frac{1}{2}}\|^2+\frac{(2N+1)}{8\sqrt{K}\sqrt{z_0}}\Delta t\sum^m_{k=1}(hg_{M_s}^{k-\frac{1}{2}})^2+K^{\frac{3}{2}}\Delta t\sum^m_{k=1} \frac{2N+1}{2\sqrt{z_0}}(Q^{k-\frac{1}{2}})^2
\\
&+\frac{\sqrt{K}}{4(a+b)}\Delta t\sum^m_{k=1}z_0^{\frac{3}{2}} \sum_{n=1}^N a_n(\frac{1}{b_n}r^{k-\frac{1}{2}})^2+\frac{K}{2}\Delta t\sum^m_{k=1}\|\bar{T}^{k-\frac{1}{2}}\|_1^2+\frac{K}{2}\Delta t\sum^m_{k=1}\|G^{k-\frac{1}{2}}\|^2\\
&+ \frac{\sqrt{K}z_0^{\frac{3}{2}}}{2}\sum_{n=1}^N a_n\Delta t\sum^m_{k=1} \left[(G^{k-\frac{1}{2}}_1)^2+(Z^{k-\frac{1}{2}}_n)^2\right].\nonumber
\end{aligned}
\end{equation}
After some trivial manipulation, we further obtain
\begin{equation}\label{eqs4_24}
\begin{aligned}
&\frac{K}{2}\|\bar{T}^{m}\|_1^2
+\sqrt{K}z_0^{\frac{3}{2}} \sum_{n=1}^N a_n \frac{1}{2}(Z_n^m)^2\\
\leq&\Delta t\sum^m_{k=1}\left[\frac{K}{2}\|\bar{T}^{k-\frac{1}{2}}\|_1^2+\frac{\sqrt{K}z_0^{\frac{3}{2}}}{2}\sum_{n=1}^N a_n(Z^{k-\frac{1}{2}}_n)^2\right]+
\frac{K}{2}\|\xi_1\|_1^2
+\frac{\Delta t}{2}\sum^m_{k=1}\|g^{k-\frac{1}{2}}\|^2+\frac{(2N+1)}{8\sqrt{K}\sqrt{z_0}}\Delta t\sum^m_{k=1}(hg_{M_s}^{k-\frac{1}{2}})^2\\
&+K^{\frac{3}{2}}\Delta t\sum^m_{k=1} \frac{2N+1}{2\sqrt{z_0}}(Q^{k-\frac{1}{2}})^2
+\frac{\sqrt{K}}{4(a+b)}\Delta t\sum^m_{k=1}z_0^{\frac{3}{2}} \sum_{n=1}^N a_n(\frac{1}{b_n}r^{k-\frac{1}{2}})^2+\frac{K}{2}\Delta t\sum^m_{k=1}\|G^{k-\frac{1}{2}}\|^2
\\
&+ \frac{\sqrt{K}z_0^{\frac{3}{2}}}{2}\sum_{n=1}^N a_n\Delta t\sum^m_{k=1} (G^{k-\frac{1}{2}}_1)^2+\frac{a}{2}\|\eta_1\|^2.
\end{aligned}
\end{equation}
It follows from Lemma \ref{Lemma2} and (\ref{eqs4_24}) that
\begin{equation}
\begin{aligned}
&\frac{K}{2}\|\bar{T}^{m}\|_1^2
+\sqrt{K}z_0^{\frac{3}{2}} \sum_{n=1}^N a_n \frac{1}{2}(Z_n^m)^2\\
\leq&\Bigg\{
\frac{K}{2}\|\xi_1\|_1^2
+\frac{\Delta t}{2}\sum^m_{k=1}\|g^{k-\frac{1}{2}}\|^2+\frac{(2N+1)}{8\sqrt{K}\sqrt{z_0}}\Delta t\sum^m_{k=1}(hg_{M_s}^{k-\frac{1}{2}})^2+K^{\frac{3}{2}}\Delta t\sum^m_{k=1} \frac{2N+1}{2\sqrt{z_0}}(Q^{k-\frac{1}{2}})^2
\\
&+\frac{\sqrt{K}}{4(a+b)}\Delta t\sum^m_{k=1}z_0^{\frac{3}{2}} \sum_{n=1}^N a_n(\frac{1}{b_n}r^{k-\frac{1}{2}})^2+\frac{K}{2}\Delta t\sum^m_{k=1}\|G^{k-\frac{1}{2}}\|^2+ \frac{\sqrt{K}z_0^{\frac{3}{2}}}{2}\sum_{n=1}^N a_n\Delta t\sum^m_{k=1} (G^{k-\frac{1}{2}}_1)^2+\frac{a}{2}\|\eta_1\|^2\Bigg\}e^D,\nonumber
\end{aligned}
\end{equation}
i.e.,
\begin{equation}
\begin{aligned}
\|\bar{T}^{m}\|_1^2
\leq&\Bigg\{
\frac{K}{2}\|\xi_1\|_1^2
+\frac{\Delta t}{2}\sum^m_{k=1}\|g^{k-\frac{1}{2}}\|^2+\frac{(2N+1)}{8\sqrt{K}\sqrt{z_0}}\Delta t\sum^m_{k=1}(hg_{M_s}^{k-\frac{1}{2}})^2+K^{\frac{3}{2}}\Delta t\sum^m_{k=1} \frac{2N+1}{2\sqrt{z_0}}(Q^{k-\frac{1}{2}})^2
\\
&+\frac{\sqrt{K}}{4(a+b)}\Delta t\sum^m_{k=1}z_0^{\frac{3}{2}} \sum_{n=1}^N a_n(\frac{1}{b_n}r^{k-\frac{1}{2}})^2+\frac{K}{2}\Delta t\sum^m_{k=1}\|G^{k-\frac{1}{2}}\|^2+ \frac{\sqrt{K}z_0^{\frac{3}{2}}}{2}\sum_{n=1}^N a_n\Delta t\sum^m_{k=1} (G^{k-\frac{1}{2}}_1)^2\\
&+\frac{a}{2}\|\eta_1\|^2\Bigg\}\frac{2e^D}{K}.\nonumber
\end{aligned}
\end{equation}
From Lemma \ref{Lemma3}, it arrives at (\ref{eqs3_30}) immediately.
\end{proof}

Now we state the stability of the developed numerical scheme. By observing Lemma \ref{Lemma4}, we can easily obtain the following stability result by replacing
$g_i^{k-\frac{1}{2}}, r^{k-\frac{1}{2}}$ with $F_i^{k-\frac{1}{2}}, R^{k-\frac{1}{2}}$, respectively, and letting $G_{i-\frac{1}{2}}^{k-\frac{1}{2}}=0.$
\begin{thm}
Assume that $\{\bar{T}^k_{i}|0\leq i\leq M_s,0\leq k\leq M_t\}$ is the solution to the numerical scheme (\ref{eqn4_5}). Then for $1\leq m\leq M_t$, one has
\begin{equation}
\begin{aligned}
\|\bar{T}^{m}\|_{\infty}^2
\leq&\Bigg\{
\frac{K}{2}\|\xi_1\|_1^2
+\frac{\Delta t}{2}\sum^m_{k=1}\|F^{k-\frac{1}{2}}\|^2+\frac{(2N+1)}{8\sqrt{K}\sqrt{z_0}}\Delta t\sum^m_{k=1}(hF_{M_s}^{k-\frac{1}{2}})^2+K^{\frac{3}{2}}\Delta t\sum^m_{k=1} \frac{2N+1}{2\sqrt{z_0}}(Q^{k-\frac{1}{2}})^2
\nonumber\\
&+\frac{\sqrt{K}}{4(a+b)}\Delta t\sum^m_{k=1}z_0^{\frac{3}{2}} \sum_{n=1}^N a_n(\frac{1}{b_n}R^{k-\frac{1}{2}})^2+\frac{a}{2}\|\eta_1\|^2\Bigg\}\frac{2e^D}{K}.
\nonumber
\end{aligned}
\end{equation}
\end{thm}

Now we employ Lemma \ref{Lemma4} to investigate the convergence of the numerical scheme (\ref{eqn4_5}).
\begin{thm}
Assume that $\bar{T}(x_i,t_k)$ and $\bar{T}_{i}^{k}\{0\leq i\leq M_s,0\leq k\leq M_t\}$ are the solutions to (\ref{eqn-250925A}) and (\ref{eqn4_5}), respectively. Let $\{\bar{e}_{i}^{k}=\bar{T}(x_i,t_{k})-\bar{T}_{i}^{k}|0\leq i\leq M_s,0\leq k\leq M_t\}$. Then we have
\begin{subequations}\label{eqn-10}
\begin{align}
\|\bar{e}^{m}\|_{\infty}^2
\leq&C(\Delta t^4+h^4).
\nonumber
\end{align}
\end{subequations}
\end{thm}
\begin{proof}
Let $
E_n^{k-\frac{1}{2}}=\zeta_n(t_{k-\frac{1}{2}})-{Z}_n^{k-\frac{1}{2}},
\epsilon_n^{k-\frac{1}{2}}=\varsigma_n(t_{k-\frac{1}{2}})-{\sigma}_n^{k-\frac{1}{2}},\bar{\omega}_{i}^{k-\frac{1}{2}}=\bar{W}(x_i,t_{k-\frac{1}{2}})-\bar{W}_{i}^{k-\frac{1}{2}},0\leq i\leq M_s,1\leq k\leq M_t.$
Note that
\begin{equation}
    \begin{aligned}
\sum_{n=1}^{N}\frac{a_n}{b_n^2}
=&\sum_{n=1}^{N}\frac{2(1-b_n)}{(2N+1)b_n^2}
\leq \sum_{n=1}^{N}\frac{2}{(2N+1)b_n^2}
\leq \frac{1}{b^2_N},\nonumber
    \end{aligned}
\end{equation}
and
\begin{equation}
    \begin{aligned}
\sum_{n=1}^{N}a_n=&\sum_{n=1}^{N}\frac{2}{2N+1}\sin^2(\frac{n\pi}{2N+1})
\leq\sum_{n=1}^{N}\frac{2}{2N+1}<1.\nonumber
    \end{aligned}
\end{equation}
Then error $e_i^k$ satisfies
\begin{equation}
\left\{\begin{aligned}
&\bar{\omega}^{k-\frac{1}{2}}_i+a\delta_t\bar{\omega}^{k-\frac{1}{2}}_i=K\left(\delta^2_x\bar{e}^{k-\frac{1}{2}}_i+b\delta^2_x\bar{\omega}^{k-\frac{1}{2}}_i\right)+P^{k-\frac{1}{2}}_i,0\leq i \leq M_s,1\leq k\leq M_t,\\
&\bar{e}^0_i=0,0\leq i \leq M_s\\
&\bar{\omega}^0_i=0,0\leq i \leq M_s\\
&\bar{e}^{k-\frac{1}{2}}_0=0,1\leq k\leq M_t,\\
&\bar{\omega}^{k-\frac{1}{2}}_{M_s}+a\delta_t\bar{\omega}^{k-\frac{1}{2}}_{M_s}=\frac{2K}{h}\Bigg\{-\delta_x\bar{e}^{k-\frac{1}{2}}_{{M_s}-\frac{1}{2}}-b\delta_{x}\bar{\omega}^{k-\frac{1}{2}}_{{M_s}-\frac{1}{2}}-\sqrt{\frac{z_0}{K}}\Bigg[\left(1+\sum_{n=1}^N \frac{a_n}{b_n}\right) \bar{\omega}^{k-\frac{1}{2}}_{M_s}-z_0 \sum_{n=1}^N \frac{a_n}{b_n} \epsilon^{k-\frac{1}{2}}_n\Bigg]\Bigg\}\nonumber\\
&+Q^{k-\frac{1}{2}}_{M_s}
,1\leq k\leq M_t, \\
&(z_0-b_n z_0)\epsilon^{k-\frac{1}{2}}_n+b_n\left[E^{k-\frac{1}{2}}_n+(a+b)\epsilon^{k-\frac{1}{2}}_n+ab\delta_t\epsilon^{k-\frac{1}{2}}_n\right]=\bar{\omega}^{k-\frac{1}{2}}_{M_s}+R^{k-\frac{1}{2}}_{M_s},1\leq k\leq M_t, \\
&E^0_n=\epsilon^0_n=0,(n=1,\dots,N),\\
&\delta_x\bar{\omega}^{k-\frac{1}{2}}_{i-\frac{1}{2}}-\delta_x\delta_t\bar{e}^{k-\frac{1}{2}}_{i-\frac{1}{2}}=G^{k-\frac{1}{2}}_{i-\frac{1}{2}},0\leq i \leq M_s,1\leq k\leq M_t,\\
&\epsilon_n^{k-\frac{1}{2}}-\delta_{t}E_n^{k-\frac{1}{2}}=G^{k-\frac{1}{2}}_1,1\leq k\leq M_t.\nonumber
\end{aligned}\right.
\end{equation}
Using Lemma \ref{Lemma4} and (\ref{eqs4_2})--(\ref{eqs4_4}), it is clear that
\begin{subequations}\label{eqn-10}
\begin{align}
\|\bar{e}^{m}\|_{\infty}^2
\leq&\Bigg\{
\frac{\Delta t}{2}\sum^m_{k=1}\|P^{k-\frac{1}{2}}\|^2+\frac{(2N+1)}{8\sqrt{K}\sqrt{z_0}}\Delta t\sum^m_{k=1}(hQ_{M_s}^{k-\frac{1}{2}})^2
+\frac{\sqrt{K}}{4(a+b)}\Delta t\sum^m_{k=1}z_0^{\frac{3}{2}} \sum_{n=1}^N a_n(\frac{1}{b_n}R_{M_s}^{k-\frac{1}{2}})^2\nonumber\\
&+\frac{K}{2}\Delta t\sum^m_{k=1}\|G^{k-\frac{1}{2}}\|^2
+ \frac{\sqrt{K}z_0^{\frac{3}{2}}}{2}\sum_{n=1}^N a_n\Delta t\sum^m_{k=1} (G^{k-\frac{1}{2}}_1)^2\Bigg\}\frac{2e^D}{K}\nonumber\\
\leq&\Bigg\{
\frac{1}{2}(\Delta t^2+h^2)^2+\frac{(2N+1)}{8\sqrt{K}\sqrt{z_0}}h^2(\Delta t^2+h)^2
+\frac{\sqrt{K}}{4(a+b)}z_0^{\frac{3}{2}}\frac{1}{b^2_N}(\Delta t^2)^2+\frac{K}{2}(\Delta t^2+h^2)^2
\nonumber\\
&+ \frac{\sqrt{K}z_0^{\frac{3}{2}}}{2} (\Delta t^2)^2\Bigg\}\frac{2C^2De^D}{K}\nonumber\\
\leq&C(\Delta t^4+h^4).\nonumber
\end{align}
\end{subequations}
\end{proof}

\section{Numerical examples}\label{Sec4}
This section presents several numerical examples to illustrate the numerical accuracy and effectiveness of the developed finite difference scheme (\ref{eqn4_5}).
To test the convergence of the proposed numerical scheme, we define the error, and the temporal and spatial convergence rates by
$$E(h,\tau)=\max\limits_{1\leq i\leq M_s}|{
T}(x_i,D)-\bar{
T}_i^{M_t}|,rate_t=log_2\frac{E(h,\tau)}{E(h,\tau/2)},rate_s=log_2\frac{E(h,\tau)}{E(h/2,\tau)},$$
where $
T(x_i,D)$ is the exact solution at $(x_i,D)$, and $\bar{
T}_i^{M_t}$ denotes the corresponding numerical solution based on spatial step-size $h$ and temporal step-size $\tau$.
\textbf{Example 1.}
For the problem (\ref{eqs1_3}), let $K=D=1, T(x,t)=
3e^{-x^2-2t}$ be the exact solution with the source function $f(x,t)=6[-1+2a-K(2x^2-1)(1-2b)]e^{-x^2-2t}$ and the initial and boundary value functions $\phi(t)=3e^{-2t},\xi(x)=3e^{-x^2},\eta(x)=-6e^{-x^2}$.

Now we employ a popular technique (e.g., as mentioned in \cite{Gao2013finite,Han2006ARTIFICIAL,Zheng2008Numerical,
Zhang2018Numerical}) to detect numerical precision.
Noticing that $e^{-x^2}, x\geq7$ is quite small and equals zero
with a tolerance less than $10^{-15}.$
Thus we set the computational domain $\Omega_2=[0,7]$, and use the developed numerical scheme to calculate the present problem. Here, two cases are considered for the values of $a$ and $b$ in (\ref{eqs1_3}), i.e.,
$a=1,b=0.2$ and $a=2,b=3$. First, we take $h=\tau$ in Table \ref{tab:1} to observe the convergence rates.
Then, $h=\tau/50$ is chosen in Table \ref{tab:2} to make the temporal error in $ E(h,\tau) $ more apparent, and $\tau=h/50$ is set in Table \ref{tab:3} to highlight the spatial error in $ E(h,\tau) $.
For these obtained numerical results in Tables \ref{tab:1}--\ref{tab:3}, it clearly
indicates that the convergence rate is  consistent with our expected second-order accuracy. At the same time, these tabular results also confirm the validity of our numerical theory.
\begin{table}[h!]
  \begin{center}
    \caption{Errors and convergence rates  for Example 1 with $h=\tau$.}
      \label{tab:1}
    \begin{tabular}{lcclcc}
        \hline
     \multirow{2}{*}{$\tau$}$\quad\quad$ & \multicolumn{2}{c}{\textbf{$a=1,b=0.2$}}& \multicolumn{3}{c}{\textbf{$a=2,b=3$}}\\
    \cline{2-3}  \cline{5-6}
       & \textbf{$E(h,\tau)$}$\quad\quad$ & \textbf{$
rate_t$}$\quad\quad$& & \textbf{$E(h,\tau)$}$\quad\quad$ & \textbf{$
rate_t$}\\
      \hline		
      1/8$\quad\quad$ & 5.170190e-03$\quad\quad$ & $\quad\quad$&
          &8.902904e-03$\quad\quad$& \\
      1/16$\quad\quad$ & 1.273184e-03$\quad\quad$ & 2.0218$\quad\quad$&
           &2.258402e-03$\quad\quad$&1.9790\\
     1/32$\quad\quad$& 3.159761e-04$\quad\quad$ & 2.0106$\quad\quad$&
         &5.687662e-04$\quad\quad$&1.9894\\
      1/64$\quad\quad$ & 7.870948e-05$\quad\quad$ & 2.0052$\quad\quad$&
         &1.427173e-04$\quad\quad$&1.9947\\
      1/128$\quad\quad$ &1.964211e-05$\quad\quad$ & 2.0026$\quad\quad$&
            &3.574537e-05$\quad\quad$&1.9973\\
          \hline
    \end{tabular}
  \end{center}
\end{table}
\begin{table}[h!]
  \begin{center}
    \caption{Errors and convergence rates for Example 1 with $h=\tau/50$}
    \label{tab:2}
     \begin{tabular}{lcclcc}
        \hline
     \multirow{2}{*}{$\tau$}$\quad\quad$ & \multicolumn{2}{c}{\textbf{$a=1,b=0.2$}}& \multicolumn{3}{c}{\textbf{$a=2,b=3$}}\\
    \cline{2-3}  \cline{5-6}
       & \textbf{$E(h,\tau)$}$\quad\quad$ & \textbf{$
rate_t$}$\quad\quad$& & \textbf{$E(h,\tau)$}$\quad\quad$ & \textbf{$
rate_t$}\\
      \hline		
      1/8$\quad\quad$ & 5.041120e-03$\quad\quad$ & $\quad\quad$&
          &9.171536e-03$\quad\quad$ & \\
      1/16$\quad\quad$ & 1.256383e-03$\quad\quad$ & 2.0045$\quad\quad$&
           &2.291941e-03$\quad\quad$&2.0006\\
     1/32$\quad\quad$& 3.138308e-04$\quad\quad$ & 2.0012$\quad\quad$&
          &5.729676e-04$\quad\quad$&2.0000\\
      1/64$\quad\quad$ & 7.843837e-05$\quad\quad$ & 2.0004$\quad\quad$&
           &1.432461e-04$\quad\quad$&2.0000\\
      1/128$\quad\quad$ & 1.960803e-05$\quad\quad$ & 2.0001$\quad\quad$&
            &3.581243e-05$\quad\quad$&2.0000\\
          \hline
    \end{tabular}
  \end{center}
\end{table}
\begin{table}[h!]
  \begin{center}
    \caption{Errors and convergence rates for Example 1 with $\tau=h/50$}
    \label{tab:3}
    \begin{tabular}{lcclcc}
        \hline
     \multirow{2}{*}{$h$}$\quad\quad$ & \multicolumn{2}{c}{\textbf{$a=1,b=0.2$}}& \multicolumn{3}{c}{\textbf{$a=2,b=3$}}\\
    \cline{2-3}  \cline{5-6}
       & \textbf{$E(h,\tau)$}$\quad\quad$ & \textbf{$
rate_s$}$\quad\quad$& & \textbf{$E(h,\tau)$}$\quad\quad$ & \textbf{$
rate_s$}\\
      \hline		
      1/8$\quad\quad$ & 1.226607e-03$\quad\quad$ & $\quad\quad$&
          &3.499495e-03$\quad\quad$ & \\
      1/16$\quad\quad$ & 3.060345e-04$\quad\quad$& 2.0029$\quad\quad$&
           &8.754240e-04$\quad\quad$&1.9991\\
     1/32$\quad\quad$& 7.644937e-05$\quad\quad$ & 2.0011$\quad\quad$&
          &2.186399e-04$\quad\quad$&2.0014\\
      1/64$\quad\quad$ & 1.911024e-05$\quad\quad$ & 2.0002$\quad\quad$&
           &5.466917e-05$\quad\quad$&1.9998\\
      1/128$\quad\quad$ & 4.777257e-06$\quad\quad$ & 2.0001$\quad\quad$&
            &1.366647e-05$\quad\quad$&2.0001\\
          \hline
    \end{tabular}
  \end{center}
\end{table}

\textbf{Example 2.}
For the problem (\ref{eqs1_3}), let $K=1, D=20,a=0.01,b=0.1,$ the source function $f(x,t)=100(1+t)e^{-60|x-0.5|}$ and the initial and
boundary value functions $\xi(x)=\eta(x)=0$.

By this example, we compare the numerical solutions obtained by solving the problem under the artificial boundary condition with those obtained based on the Dirichlet boundary conditions. For the considered problem, we assume that the considered problem satisfies homogeneous Dirichlet boundary conditions on space domain $[0,20]$. A numerical scheme is then obtained by replacing (\ref{eqs4_4.5e})--(\ref{eqs4_4.5g}) and (\ref{eqs4_4.5i}) with $\bar{T}_{M_s}^{k-1/2}=0,1\leq k\leq M_t$ in  (\ref{eqn4_5}), and the corresponding problem is numerically solved with $h=\tau=1/64$ (Scheme I).
Concurrently, we establish the artificial boundary based on space domain $[0,1]$
and perform calculations with $h=\tau=1/64$ (Scheme II). Figures \ref{fig:1} and \ref{fig:2} depict the curves of Scheme I and Scheme II at $t=14,16,18,20$, respectively. By observing these figures, we find that the temperature $T$ reaches its peak in the space domain $[0,1]$, and the numerical solutions under Scheme I and Scheme II are almost identical on domain $[0,1]$. Furthermore, from these figures, we also observe that to ensure $T(x_r,t)$ is approximately zero (i.e., to satisfy the homogeneous Dirichlet boundary condition for the considered problem), the computational domain $[0,x_r]$ for the corresponding problem should ideally satisfy $x\geq15$. Consequently, this suggests that when one focuses on temperature $T$ such as its peak value, the artificial boundary method offers a significant advantage due to its efficient computational speed resulting from a smaller computational domain.
\begin{figure}[htbp]
\centering
\includegraphics[width=0.45\columnwidth]{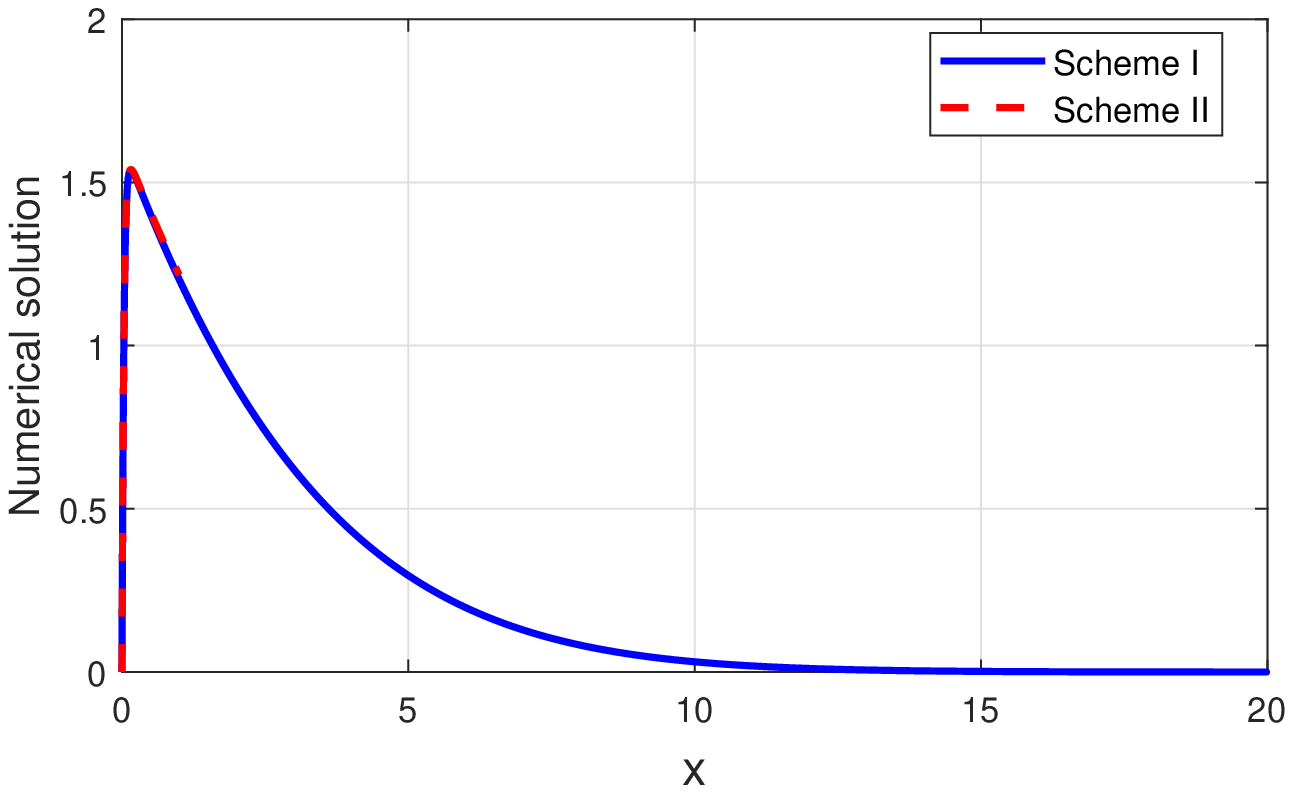}
\includegraphics[width=0.45\columnwidth]{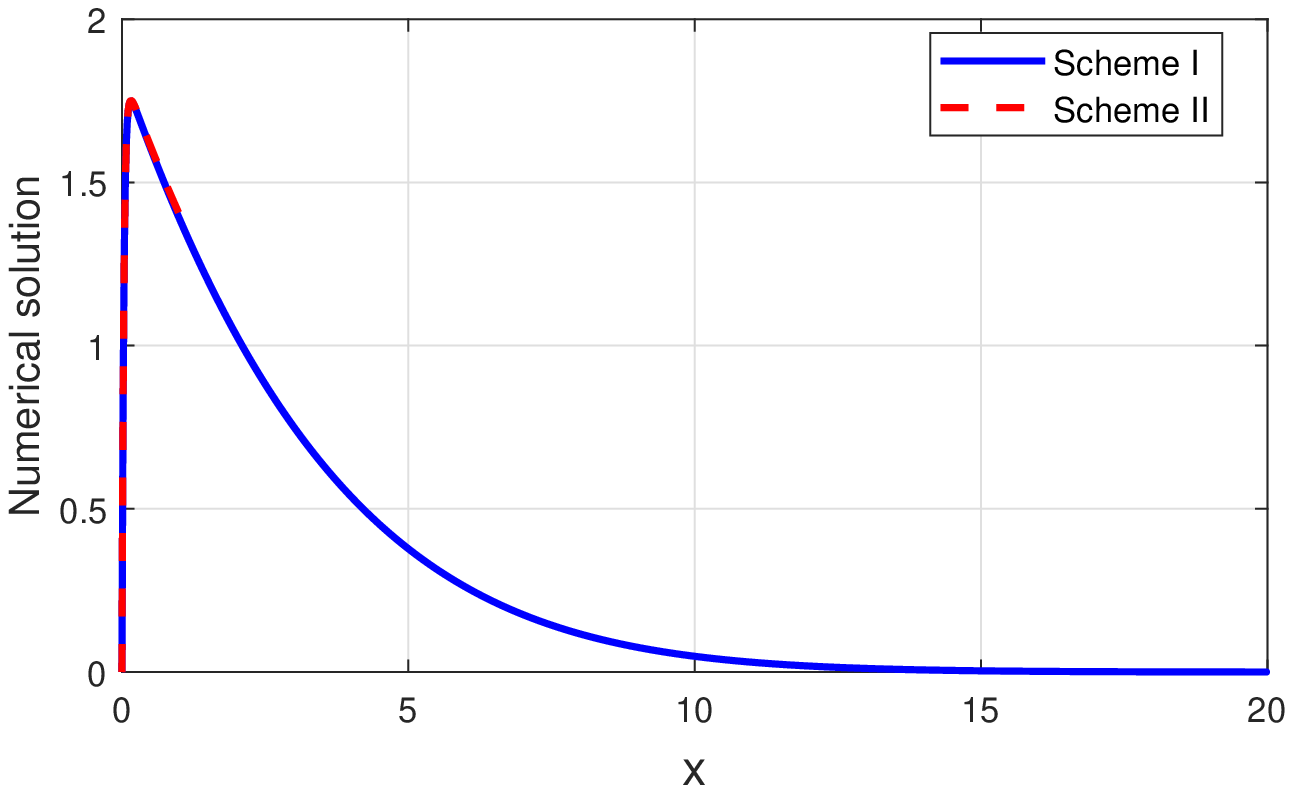}
\caption{Numerical solutions at $t=14$ (left) and $t=16$ (right) for Scheme I and Scheme II with $h=\tau=1/64$.}
\label{fig:1}
\end{figure}
\begin{figure}[htbp]
\centering
\includegraphics[width=0.45\columnwidth]{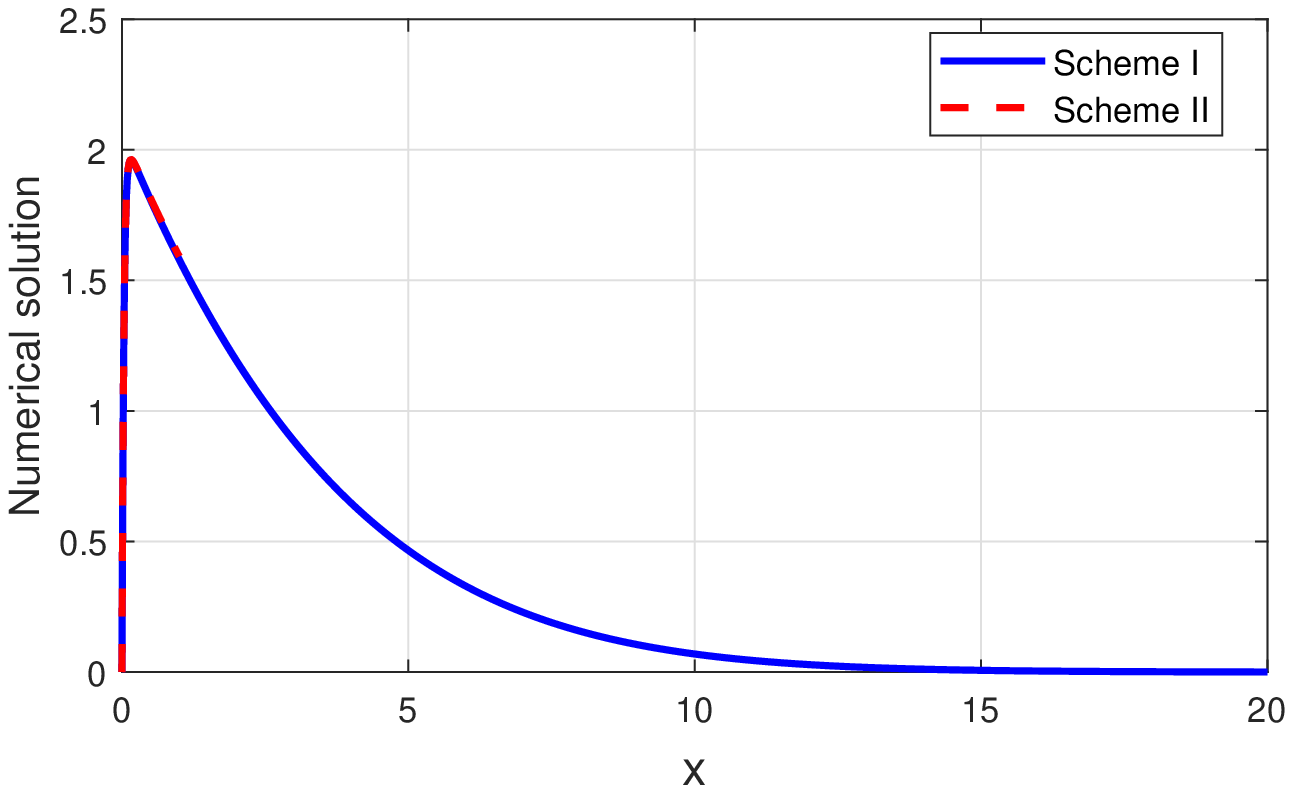}
\includegraphics[width=0.45\columnwidth]{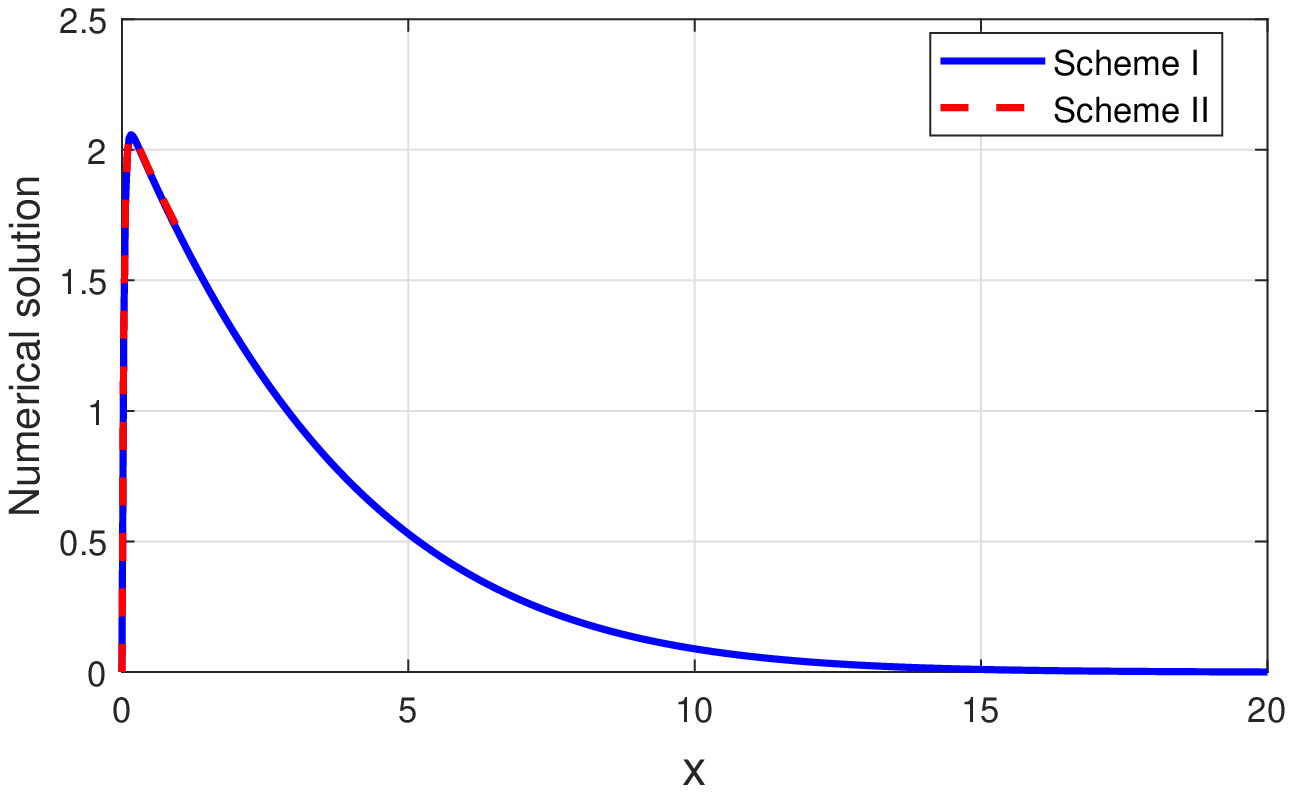}
\caption{Numerical solutions at $t=18$ (left) and $t=20$ (right) for Scheme I and Scheme II with $h=\tau=1/64$.}
\label{fig:2}
\end{figure}

\section{Conclusion}\label{Sec6}
This paper proposes an efficient numerical method for solving a DPL heat conduction equation on an unbounded domain. To transform this problem on the unbounded domain into an initial-boundary value problem on a bounded computational domain, the artificial boundary method is introduced, and high-order local artificial boundary is constructed by utilizing
the Pad\'e approximation and the Laplace transform. Furthermore, for the resulting reduced problem on the bounded computational domain equipped with high-order local artificial boundary, its stability in the $L^2$-norm is discussed. To facilitate the construction of a finite difference scheme for the reduced problem, an auxiliary variable is introduced to reduce the order of the time derivative in the governing equation. Regarding the developed numerical scheme, a rigorous stability and error analysis is conducted, demonstrating its unconditional stability and second-order accuracy in both temporal and spatial directions. Finally, two numerical examples are presented to validate the correctness of the developed numerical theory and the effectiveness of the numerical scheme.

\section*{Acknowledgements}
The first researcher is supported by the Research Foundation of Education Commission of Hunan Province of China (Grant No. 23A0126) and the 111 Project (Grant No. D23017).

\end{document}